\batchmode
\makeatletter
\def\input@path{{"C:/Users/Ori Segel/Dropbox/Math/Project/"}}
\makeatother
\documentclass[11pt,oneside]{amsart}
\usepackage[latin9]{inputenc}
\usepackage{color}
\usepackage{amstext}
\usepackage{amsthm}
\usepackage{amssymb}
\usepackage{stackrel}
\usepackage[pdftex,unicode=true,pdfusetitle,
 bookmarks=true,bookmarksnumbered=false,bookmarksopen=false,
 breaklinks=false,pdfborder={0 0 0},pdfborderstyle={},backref=page,colorlinks=true]
 {hyperref}

\makeatletter
\theoremstyle{plain}
\newtheorem{thm}{\protect\theoremname}[section]
\theoremstyle{plain}
\newtheorem{fact}[thm]{\protect\factname}
\theoremstyle{definition}
\newtheorem{defn}[thm]{\protect\definitionname}
\theoremstyle{remark}
\newtheorem{rem}[thm]{\protect\remarkname}
\theoremstyle{definition}
\newtheorem{example}[thm]{\protect\examplename}
\theoremstyle{plain}
\newtheorem{lem}[thm]{\protect\lemmaname}
\theoremstyle{plain}
\newtheorem{cor}[thm]{\protect\corollaryname}
\theoremstyle{plain}
\newtheorem{prop}[thm]{\protect\propositionname}
\theoremstyle{remark}
\newtheorem{claim}[thm]{\protect\claimname}

\usepackage{amsfonts}
\usepackage{nicefrac}
\usepackage{amscd}
\usepackage{a4wide}
\usepackage{url}

\makeatother

\providecommand{\claimname}{Claim}
\providecommand{\corollaryname}{Corollary}
\providecommand{\definitionname}{Definition}
\providecommand{\examplename}{Example}
\providecommand{\factname}{Fact}
\providecommand{\lemmaname}{Lemma}
\providecommand{\propositionname}{Proposition}
\providecommand{\remarkname}{Remark}
\providecommand{\theoremname}{Theorem}

\begin{document}
\global\long\def\p{\mathbf{p}}%
\global\long\def\q{\mathbf{q}}%
\global\long\def\C{\mathfrak{C}}%
\global\long\def\SS{\mathcal{P}}%
 
\global\long\def\pr{\operatorname{pr}}%
\global\long\def\image{\operatorname{im}}%
\global\long\def\otp{\operatorname{otp}}%
\global\long\def\dec{\operatorname{dec}}%
\global\long\def\suc{\operatorname{suc}}%
\global\long\def\pre{\operatorname{pre}}%
\global\long\def\qe{\operatorname{qf}}%
 
\global\long\def\ind{\operatorname{ind}}%
\global\long\def\Nind{\operatorname{Nind}}%
\global\long\def\lev{\operatorname{lev}}%
\global\long\def\Suc{\operatorname{Suc}}%
\global\long\def\HNind{\operatorname{HNind}}%
\global\long\def\minb{{\lim}}%
\global\long\def\concat{\frown}%
\global\long\def\cl{\operatorname{cl}}%
\global\long\def\tp{\operatorname{tp}}%
\global\long\def\id{\operatorname{id}}%
\global\long\def\cons{\left(\star\right)}%
\global\long\def\qf{\operatorname{qf}}%
\global\long\def\ai{\operatorname{ai}}%
\global\long\def\dtp{\operatorname{dtp}}%
\global\long\def\acl{\operatorname{acl}}%
\global\long\def\nb{\operatorname{nb}}%
\global\long\def\limb{{\lim}}%
\global\long\def\leftexp#1#2{{\vphantom{#2}}^{#1}{#2}}%
\global\long\def\intr{\operatorname{interval}}%
\global\long\def\atom{\emph{at}}%
\global\long\def\I{\mathfrak{I}}%
\global\long\def\uf{\operatorname{uf}}%
\global\long\def\ded{\operatorname{ded}}%
\global\long\def\Ded{\operatorname{Ded}}%
\global\long\def\Df{\operatorname{Df}}%
\global\long\def\Th{\operatorname{Th}}%
\global\long\def\eq{\operatorname{eq}}%
\global\long\def\Aut{\operatorname{Aut}}%
\global\long\def\ac{ac}%
\global\long\def\DfOne{\operatorname{df}_{\operatorname{iso}}}%
\global\long\def\modp#1{\pmod#1}%
\global\long\def\sequence#1#2{\left\langle #1\left|\,#2\right.\right\rangle }%
\global\long\def\set#1#2{\left\{  #1\left|\,#2\right.\right\}  }%
\global\long\def\Diag{\operatorname{Diag}}%
\global\long\def\Nn{\mathbb{N}}%
\global\long\def\mathrela#1{\mathrel{#1}}%
\global\long\def\twiddle{\mathord{\sim}}%
\global\long\def\mathordi#1{\mathord{#1}}%
\global\long\def\Qq{\mathbb{Q}}%
\global\long\def\dense{\operatorname{dense}}%
 
\global\long\def\cof{\operatorname{cof}}%
\global\long\def\tr{\operatorname{tr}}%
\global\long\def\treeexp#1#2{#1^{\left\langle #2\right\rangle _{\tr}}}%
\global\long\def\x{\times}%
\global\long\def\forces{\Vdash}%
\global\long\def\Vv{\mathbb{V}}%
\global\long\def\Uu{\mathbb{U}}%
\global\long\def\tauname{\dot{\tau}}%
\global\long\def\ScottPsi{\Psi}%
\global\long\def\cont{2^{\aleph_{0}}}%
\global\long\def\MA#1{{MA}_{#1}}%
\global\long\def\rank#1#2{R_{#1}\left(#2\right)}%
\global\long\def\cal#1{\mathcal{#1}}%

\def\Ind#1#2{#1\setbox0=\hbox{$#1x$}\kern\wd0\hbox to 0pt{\hss$#1\mid$\hss} \lower.9\ht0\hbox to 0pt{\hss$#1\smile$\hss}\kern\wd0} 
\def\Notind#1#2{#1\setbox0=\hbox{$#1x$}\kern\wd0\hbox to 0pt{\mathchardef \nn="3236\hss$#1\nn$\kern1.4\wd0\hss}\hbox to 0pt{\hss$#1\mid$\hss}\lower.9\ht0 \hbox to 0pt{\hss$#1\smile$\hss}\kern\wd0} 
\def\nind{\mathop{\mathpalette\Notind{}}} 

\global\long\def\ind{\mathop{\mathpalette\Ind{}}}%
 
\global\long\def\nind{\mathop{\mathpalette\Notind{}}}%
\global\long\def\average#1#2#3{Av_{#3}\left(#1/#2\right)}%
\global\long\def\Ff{\mathfrak{F}}%
\global\long\def\mx#1{Mx_{#1}}%
\global\long\def\maps{\mathfrak{L}}%

\global\long\def\Esat{E_{\mbox{sat}}}%
\global\long\def\Ebnf{E_{\mbox{rep}}}%
\global\long\def\Ecom{E_{\mbox{com}}}%
\global\long\def\BtypesA{S_{\Bb}^{x}\left(A\right)}%
\global\long\def\supp{\operatorname{supp}}%

\global\long\def\init{\trianglelefteq}%
\global\long\def\fini{\trianglerighteq}%
\global\long\def\Bb{\cal B}%
\global\long\def\Rr{\mathbb{R}}%
\global\long\def\ord{\mathbf{ord}}%

\title{Boolean types in dependent theories}
\author{Itay Kaplan, Ori Segel and Saharon Shelah }
\thanks{The first and second author would like to thank the Israel Science
foundation for partial support of this research (Grant nos. 1533/14
and 1254/18). }
\thanks{The third author would like to thank the Israel Science Foundation
grant no: 1838/19 and the European Research Council grant 338821.
No. 1172 on the third author's list of publications.}
\address{Itay Kaplan \\
The Hebrew University of Jerusalem\\
Einstein Institute of Mathematics \\
Edmond J. Safra Campus, Givat Ram\\
Jerusalem 91904, Israel}
\email{kaplan@math.huji.ac.il}
\urladdr{math.huji.ac.il/\textasciitilde kaplan}
\address{Ori Segel \\
The Hebrew University of Jerusalem\\
Einstein Institute of Mathematics \\
Edmond J. Safra Campus, Givat Ram\\
Jerusalem 91904, Israel}
\email{ori.segel@mail.huji.ac.il}
\address{Saharon Shelah\\
The Hebrew University of Jerusalem\\
Einstein Institute of Mathematics \\
Edmond J. Safra Campus, Givat Ram\\
Jerusalem 91904, Israel}
\address{Saharon Shelah \\
Department of Mathematics\\
Hill Center-Busch Campus\\
Rutgers, The State University of New Jersey\\
110 Frelinghuysen Road\\
Piscataway, NJ 08854-8019 USA}
\email{shelah@math.huji.ac.il}
\urladdr{shelah.logic.at}
\begin{abstract}
The notion of a complete type can be generalized in a natural manner
to allow assigning a value in an arbitrary Boolean algebra $\mathcal{B}$
to each formula. We show some basic results regarding the effect of
the properties of $\mathcal{B}$ on the behavior of such types, and
show they are particularity well behaved in the case of NIP theories.
In particular, we generalize the third author's result about counting
types, as well as the notion of a smooth type and extending a type
to a smooth one. We then show that Keisler measures are tied to certain
Boolean types and show that some of the results can thus be transferred
to measures --- in particular, giving an alternative proof of the
fact that every measure in a dependent theory can be extended to a
smooth one. We also study the stable case. We consider this paper
as an invitation for more research into the topic of Boolean types.
\end{abstract}

\maketitle

\section{Introduction}

A complete type over a set $A$ in the variable $x$ is a maximal
consistent set of formulas from $L_{x}\left(A\right)$, the set of
formulas in $x$ with parameters from $A$. Of course, one can think
of $L_{x}\left(A\right)$ as a Boolean algebra by identifying formulas
which define the same set in $\C$ (the monster model). Viewed this
way, a type can also be defined as a homomorphism of Boolean algebras
from $L_{x}\left(A\right)$ to the Boolean algebra $2$. The idea
behind this work is to generalize this definition by allowing an arbitrary
Boolean algebra $\cal B$ in the range. We call these homomorphisms
\emph{$\cal B$-types over $A$} (see Definition \ref{def:generalized types}).
Without any assumptions, Boolean types may behave very wildly, but
it turns out that if the ambient theory $T$ is dependent (NIP) then
there are some restrictions on their behavior which gives some credence
to the claim that this is the right context to study such types in
full generality. 

Let us consider an example. Suppose that $T$ is the theory of the
random graph in the language $\left\{ R\right\} $ and that $\cal B$
is any Boolean algebra. Let $M\models T$ be of size $\geq\left|\cal B\right|$
and let $h:M\to\cal B$ be a surjective map. Let $p:L_{x}\left(M\right)\to\cal B$
be the unique homomorphism defined by mapping formulas of the form
$x\mathrela Ra$ to $h\left(a\right)$ and formulas of the form $x=a$
to $0$ (such a homomorphism exists by quantifier elimination and
\cite[Proposition 5.6]{MR991565}). Thus for any $\cal B$ there is
a type whose image has size $\left|\cal B\right|$. 

However, when $T$ is NIP (dependent) this fails. For example, suppose
that $T=DLO$ (the theory of $\left(\Qq,<\right)$), and suppose that
$\cal B$ is a Boolean algebra with c.c.c (the countable chain condition:
if $\sequence{\mathfrak{a}_{i}}{i<\aleph_{1}}$ are non-zero elements
then for some $i,j<\aleph_{1}$, $\mathfrak{a}_{i}\wedge\mathfrak{a}_{j}\neq0$),
for example the algebra of measurable sets up to measure 0 in some
(real) probability space. Let $M$ be any model and let $p:L_{x}\left(M\right)\to\cal B$
be any $\cal B$-type. Suppose that the image of $p$ has size $\left(2^{\aleph_{0}}\right)^{+}$.
By quantifier elimination, the image of $p$ is the algebra generated
by elements of the form $p\left(x<a\right)$ for $a\in M$ (since
$x=a$ is equivalent to $\neg\left(x<a\land a<x\right)$). It follows
that $\left|\set{p\left(x<a\right)}{a\in M}\right|=\left(2^{\aleph_{0}}\right)^{+}$,
so we can find $\sequence{a_{i}}{i<\left(2^{\aleph_{0}}\right)^{+}}$
such that $p\left(x<a_{i}\right)\neq p\left(x<a_{j}\right)$ for $i\neq j$.
By Erd\H{o}s--Rado, we may assume that $\sequence{a_{i}}{i<\omega_{1}}$
is either increasing of decreasing. Assume the former. Then $p\left(x<a_{i}\right)<^{\cal B}p\left(x<a_{j}\right)$
so $p\left(a_{i}\leq x\right)\cdot p\left(x<a_{j}\right)=p\left(a_{i}\leq x<a_{j}\right)\neq0$
for all $i<j<\omega_{1}$ and thus $\set{p\left(a_{i}\leq x<a_{i+1}\right)}{i<\omega_{1}}$
is a set of size $\aleph_{1}$ of nonzero mutually disjoint elements
from $\cal B$, contradiction. This boundedness of the image generalizes
to any NIP theory $T$, see Proposition \ref{prop:image not too big}
below. 

Let us consider another example. In the classical settings, any type
can be realized in an elementary extension. Once a type is realized,
it has a unique extension to any model. Boolean types that have unique
extensions are called \emph{smooth}. Going back to the theory of the
random graph, let us assume that $\cal B$ is the algebra of Borel
subsets of $2^{\kappa}$ up to measure $0$ (with the measure $\mu$
being the product measure, see \cite[254J]{MR2462280}), and assume
that $\left|M\right|\geq\kappa$. Let $h:M\to\kappa$ be surjective.
Let $p:L_{x}\left(M\right)\to\Bb$ be defined by $p\left(x\mathrela Ra\right)=U_{h\left(a\right)}=\set{\eta\in2^{\kappa}}{\eta\left(h\left(a\right)\right)=1}$
so that its measure is $1/2$ (to be precise we should put the class
of this set, but we will ignore this nuance for this discussion) and
$p\left(x=a\right)=0$ for any $a\in M$ (again such a $\Bb$-type
exists). Note that if $N\succ M$ and $q$ is a $\Bb$-type extending
$p$, then $q\left(x=a\right)=0$ for any $a\in N$ since otherwise
it has positive measure, which leads to a contradiction (since for
any conjunction $\varphi$ of atomic formulas or their negations over
$M$ which $a$ realizes satisfies $\mu\left(q\left(\varphi\right)\right)\geq\mu\left(q\left(x=a\right)\right)$
and the left hand side tends to 0 as the number of distinct conditions
in $\varphi$ grows by the choice of $p$). Thus by Sikorski (see
Fact \ref{fact:function-extension}), we have a lot of freedom in
extending $q$ to any $N'\succ N$. Hence no smooth extension of $p$
exists. Again, when $T$ is NIP and the Boolean algebra is nice enough,
every type has a smooth extension, see e.g. Corollary \ref{cor:smooth-extension}
below. 

\subsection{Structure of the paper}

In Section \ref{sec:Boolean-types} we prove all the main results
on Boolean types. In Subsection \ref{subsec:Basic-Definition} we
give the basic definitions. In particular we define two kinds of maps
between Boolean types: those induced by elementary maps (like in the
classical setting) and those induced by embeddings of the Boolean
algebra itself. In Subsection \ref{subsec:Counting-Boolean-types}
we then give bounds to the number of Boolean types up to conjugation
(in particular generalizing a result of the third author \cite[Theorem 5.21]{Sh950})
or just image conjugation. In Subsection \ref{subsec:Smooth-Boolean-types}
we define and discuss smooth Boolean types as well as a stronger notion,
that of a realized Boolean type. We prove that under NIP, if the algebra
is complete, smooth types exist. 

In Section \ref{sec:Relation-to-Keisler} we relate Boolean types
to Keisler measures, i.e., finitely additive measures on definable
sets. In Subsection \ref{subsec:Counting-Keisler-Measures} we apply
the results of Subsection \ref{subsec:Counting-Boolean-types} to
count Keisler measures up to conjugation (we also give a more direct
proof, using the VC-theorem). In Subsection \ref{subsec:Smooth-Keisler-Measures}
we give an alternative proof to the well-known fact that every Keisler
measure extends to a smooth one \cite[Proposition 7.9]{pierrebook},
using the results of Subsection \ref{subsec:Smooth-Boolean-types}.

In Section \ref{sec:Analysis of stable case} we analyze the case
where the theory (or just one formula) is stable, as well as the totally
transcendental case, showing that in this case Boolean types are locally
averages of types (and in the t.t. case this is true for complete
types as well). 

Throughout, let $T$ be a complete first order theory. Most of the
time, we will only deal with dependent $T$. We use standard notations,
e.g., $\C$ is a monster model for $T$. As usual, all sets and models
are subsets or elementary substructures of $\C$ of cardinality $<\left|\C\right|$.

\subsubsection*{Acknowledgments}

We would like to thank Artem Chernikov and David Fremlin for giving
some comments on private communication. In particular we like to thank
Chernikov for pointing out the alternative proof of Proposition \ref{prop:keisler-measures-in-NIP}
and Fremlin for pointing out the proof for Proposition \ref{prop:extending-partial-isomorphisms}.
We would also like to thank the referee for their careful reading
and many comments which greatly helped us to improve the paper.

\section{\label{sec:Boolean-types}Boolean types}

\subsection{\label{subsec:Basic-Definition}Basic Definitions}

In this subsection we define Boolean types, which are the  basic objects
studied in this paper. We also define when two such types are conjugate
to each other. We finish this subsection by briefly discussing the
algebraic properties of Boolean types.

Let us start by recounting some standard notation for Boolean algebras.

Let $\mathcal{B}$ be a Boolean algebra, and denote by $0$ and $1$
the distinguished elements of $\mathcal{B}$ corresponding to $\bot$
and $\top$ for formulas. We denote by $\mathcal{B}^{+}$ the set
of all nonzero elements of $\mathcal{B}$.

Let $\mathfrak{a},\mathfrak{b}\in\mathcal{B}$. We denote by $-\mathfrak{a}$,
$\mathfrak{a}+\mathfrak{b}$ and $\mathfrak{a}\cdot\mathfrak{b}$
the complement, sum and product --- corresponding to $\neg,\vee$
and $\wedge$ for formulas, respectively (for example, $-0=1$). We
also write $\mathfrak{a}-\mathfrak{b}=\mathfrak{a}\cdot\left(-\mathfrak{b}\right)$.
We say $\mathfrak{a},\mathfrak{b}$ are disjoint if $\mathfrak{a}\cdot\mathfrak{b}=0$.
We also write $\left(-1\right)\cdot\mathfrak{a}$ for $-\mathfrak{a}$
and $1\cdot\mathfrak{a}$ for $\mathfrak{a}$. 

Every Boolean algebra has a canonical order relation: $\mathfrak{a}\leq\mathfrak{b}$
iff $\mathfrak{a}\cdot\mathfrak{b}=\mathfrak{a}$ or equivalently
$\mathfrak{a}+\mathfrak{b}=\mathfrak{b}$. This corresponds to $\rightarrow$
for formulas. Recall that $0$ and $1$ are the minimum and maximum
with respect to this order. 

If the supremum of a set $A\subseteq\mathcal{B}$ exists we denote
it by $\sum A$, likewise, if the infimum exists we denote it by $\prod A$.
An algebra is complete if both always exist.

A ($\kappa$-)complete subalgebra of $\mathcal{B}$ is some subalgebra
$\mathcal{B}'\subseteq\mathcal{B}$ such that if $A\subseteq\mathcal{B}'$
(and $\left|A\right|<\kappa$), and if $\sum A$ exists in $\mathcal{B}$,
then $\sum A\in\mathcal{B}'$.

If $\mathfrak{b}\in\mathcal{B}^{+}$, we define the relative algebra
$\mathcal{B}|_{\mathfrak{b}}$ - its universe is $\left\{ \mathfrak{a}\in\mathcal{B}\mid\mathfrak{a}\leq\mathfrak{b}\right\} $,
with $+,\cdot$ and $0$ inherited from $\mathfrak{\mathcal{B}}$,
and with $1$$_{\mathcal{B}|_{\mathfrak{b}}}=\mathfrak{b}$ and $\left(-\mathfrak{a}\right)_{\mathcal{B}|_{\mathfrak{b}}}$
being $\mathfrak{b}-\mathfrak{a}$. Note that $\mathcal{B}|_{\mathfrak{b}}$
is complete if $\mathcal{B}$ is.

There is a natural homomorphism $\pi_{\mathfrak{b}}:\mathcal{B}\rightarrow\mathcal{B}|_{\mathfrak{b}}$
given by $\pi_{\mathfrak{b}}\left(\mathfrak{a}\right)=\mathfrak{a}\cdot\mathfrak{b}$.

Since we will deal with homomorphisms of Bolean algebras, we will
also need the following facts (Sikorski's extension theorem):
\begin{fact}
\label{fact:function-extension} \cite[Theorem 5.5 and Theorem 5.9]{MR991565}
Assume $\mathcal{B}$ is a complete Boolean algebra and $\mathcal{A}$
is any Boolean algebra. Assume $A\subseteq\mathcal{A}$ a subset and
$f:A\rightarrow\mathcal{B}$ is a function.

Then there is a homomorphism $g:\mathcal{A}\rightarrow\mathcal{B}$
extending $f$ iff, for any $\mathfrak{a}_{0},...,\mathfrak{a}_{n-1}\in A$
and $\varepsilon_{0},...,\varepsilon_{n-1}\in\left\{ \pm1\right\} $
such that $\stackrel[i<n]{}{\prod}\varepsilon_{i}\mathfrak{a}_{i}=0$,
we also have $\stackrel[i<n]{}{\prod}\varepsilon_{i}f\left(\mathfrak{a}_{i}\right)=0$.
\end{fact}

\begin{fact}
\label{fact:homomorphism-extension} \cite[Proposition 5.8 and Theorem 5.9]{MR991565}
Assume $\mathcal{B}$ is a complete Boolean algebra and $\mathcal{A}$
is any Boolean algebra. Assume $\mathcal{A}'\subseteq\mathcal{A}$
is a subalegbra, and $f:\mathcal{A}'\rightarrow\mathcal{B}$ is a
homomorphism.

Assume further that $\mathfrak{a}\in\mathcal{A}$, $\mathfrak{b}\in\mathcal{B}$.
Then there exists a homomorphism $g:\mathcal{A}\rightarrow\mathcal{B}$
extending $f$ such that $g\left(\mathfrak{a}\right)=\mathfrak{b}$
iff $\stackrel[\mathfrak{a}'\in\mathcal{A}',\mathfrak{a}'\leq\mathfrak{a}]{}{\sum}f\left(\mathfrak{a}'\right)\leq\mathfrak{b}\leq\stackrel[\mathfrak{a}'\in\mathcal{A}',\mathfrak{a}\leq\mathfrak{a}']{}{\prod}f\left(\mathfrak{a}'\right)$.
\end{fact}

\begin{defn}
\label{def:generalized types}Suppose $\Bb$ is a Boolean algebra
and $\mathfrak{C}$ is the monster model of a complete theory $T$. 

For a set $A\subseteq\mathfrak{C}$, a \emph{(complete) $\Bb$-type
in $x$ over $A$} is a Boolean algebra homomorphism from the algebra
of formulas $L_{x}\left(A\right)$ consisting of formulas in $x$
over $A$ up to equivalence in $\C$ to $\Bb$. By slight abuse of
notation, we will use a formula to refer to its equivalence class
in $L_{x}\left(A\right)$. The set of all complete $\Bb$-types in
$x$ over $A$ is denoted by $S_{\Bb}^{x}\left(A\right)$. The set
$S_{\Bb}^{n}\left(A\right)$ for $n$ a natural number, or an ordinal,
will be the set of all complete $\Bb$-types over $A$ in some $n$
fixed variables. 

An elementary permutation $\pi:A\rightarrow A$ is a bijective elementary
map. The group of elementary permutations $\pi:A\to A$ acts on $\BtypesA$
by $\left(\pi*p\right)\left(\varphi\left(x,a\right)\right)=p\left(\varphi\left(x,\pi^{-1}\left(a\right)\right)\right)$.
Say that $p_{1},p_{2}\in\BtypesA$ are \emph{elementarily conjugates
over $A$} if they are in the same orbit. 

We say that two $\mathcal{B}$-types are \emph{image conjugates} if
there is some partial isomorphism of Boolean algebras $\sigma$ whose
domain contains the image of $p_{2}$ such that $p_{1}=\sigma\circ p_{2}$.

Finally, we say that $p_{1},p_{2}\in\BtypesA$ are \emph{conjugates
over $A$ }if there is some $\pi:A\to A$ as above, and some partial
Boolean isomorphism $\sigma$ as above such that $p_{1}=\sigma\circ\left(\pi*p_{2}\right)$.
\end{defn}

\begin{rem}
Note that $\sigma\circ\left(\pi*p\right)=\pi\ast\left(\sigma\circ p\right)$.
Thus, as image conjugation and elementary conjugation are clearly
equivalence relations, so is conjugation.
\end{rem}

\begin{rem}
Note that when $\Bb=2$ (that is $\left\{ 0,1\right\} $), a complete
$\Bb$-type is the same as a complete type.

Also, the two notions of elementarily conjugation and conjugation
identify in this case. 
\end{rem}

\begin{rem}
Note that $L_{x}\left(A\right)$ is isomorphic to the quotient of
the Lindenbaum-Tarski algebra $\mathcal{L}_{A,x}$ (the algebra of
$L$ formulas over $A$ in $x$ up to logical equivalence, see \cite[Chapter 26]{MR991595})
by the filter $F$ generated by all sentences $\varphi$ that hold
in $\mathfrak{C}$ (where parameters from $A$ are considered to be
constants).

Indeed, let $f:\mathcal{L}_{A,x}\rightarrow L_{x}\left(A\right)$
be the canonical map sending a formula to its equivalence class in
$\mathfrak{C}$. We want to show that $f^{-1}\left(\top\right)=F$.

$\psi\left(\mathfrak{C}\right)=\mathfrak{C}^{x}$ iff $\mathfrak{C}\vDash\forall x\left(\psi\left(x\right)\right)$; 

and since $\vdash\left(\forall x\psi\left(x\right)\right)\rightarrow\psi\left(x\right),$
we get $\forall x\psi\left(x\right)\leq\psi\left(x\right)$ in $\mathcal{L}_{A,x}$
thus $\psi\left(x\right)\in F$.

On the other hand if $\psi\left(x\right)\in F$ then for some sentence
$\varphi$ such that $\mathfrak{C}\vDash\varphi$ we have $\varphi\leq\psi\left(x\right)$
in $\mathcal{L}_{A,x}$, that is $\vdash\varphi\rightarrow\psi\left(x\right)$,
and certainly $\psi\left(\mathfrak{C}\right)=\mathfrak{C}^{x}$.

This means that for any Boolean algebra $\mathcal{B},$ $\left\{ f\in Hom\left(\mathcal{L}_{A,x},\mathbb{\mathcal{B}}\right)\mid f\left[F\right]=1\right\} $
are canonically isomorphic to $Hom\left(L_{x}\left(A\right),\mathcal{B}\right)$.

Thus we can define $S_{\mathcal{B}}^{x}\left(A\right)$ to be $\left\{ p\in Hom\left(\mathcal{L}_{A,x},\mathcal{B}\right)\mid F\subseteq p^{-1}\left(1\right)\right\} $
without changing anything.
\end{rem}

\begin{example}
\label{ex:set-algebra-types}One might wonder if there are any natural
examples of Boolean types.

Let $\left\langle \mathcal{B}_{i}\right\rangle _{i\in I}$ a sequence
of Boolean algebras. Then the product algebra $\stackrel[i\in I]{}{\prod}\mathcal{B}_{i}$
is defined in the usual sense of products of algebraic structures
--- its elements are choice functions, and operations are preformed
coordinatewise.

By the universal property of product algebras (see \cite[Proposition 6.3]{MR991565}),
\[
S_{\stackrel[i\in I]{}{\prod}\mathcal{B}_{i}}^{x}\left(A\right)\cong\stackrel[i\in I]{}{\prod}S_{\mathcal{B}_{i}}^{x}\left(A\right).
\]

The correspondence is given by $\left\langle p_{i}\right\rangle _{i\in I}\mapsto\left(\varphi\mapsto\left\langle p_{i}\left(\varphi\right)\right\rangle _{i\in I}\right)$.

In particular, for any cardinal $\lambda$, $S_{2^{\lambda}}^{x}\left(A\right)\cong\left(S_{2}^{x}\left(A\right)\right)^{\lambda}=\left(S^{x}\left(A\right)\right)^{\lambda}$
naturally --- where the sequence $\left\langle p_{i}\right\rangle _{i<\lambda}$
corresponds to the Boolean type $p$ satisfying $\varphi\in p_{i}\Longleftrightarrow p\left(\varphi\right)_{i}=1$.

Thus we can consider a sequence of $\lambda$ complete types to be
the same thing, for all intents as purposes, as a $\mathcal{B}$-type
for $\mathcal{B}=2^{\lambda}$. 

This case will give us an idea about the behavior of general Boolean
types.
\end{example}

\subsection{\label{subsec:Counting-Boolean-types}Counting Boolean types}

The main result of this section is Corollary \ref{cor:counting} (generalizing
a result by the third author \cite[Theorem 5.21]{Sh950}), which says
that when $T$ is NIP, the number of Boolean types over a saturated
model up to conjugation has a bound dependent not on $\left|\mathcal{B}\right|$
but on the chain condition satisfied by $\mathcal{B}$. The strategy
will be to encode Boolean types up to conjugacy as complete types
in a long variable tuple.

Throughout this subsection, $A$ is a subset of $\mathfrak{C}$.

\subsubsection{\label{subsec:count-conjugation}Counting Boolean types up to conjugation}

We first concern ourselves with the question of the number of Boolean
types up to conjugation. 
\begin{lem}
\label{lem:counting-set-algebra-types}There exists an injection $f:S_{2^{\lambda}}^{x}\left(A\right)\rightarrow S^{\left|x\right|\cdot\lambda}\left(A\right)$
such that if $f\left(p_{1}\right)$ and $f\left(p_{2}\right)$ are
conjugates then $p_{1}$ and $p_{2}$ are elementary conjugates.
\end{lem}

\begin{proof}
Note first that $\left(S^{x}\left(A\right)\right)^{\lambda}$ can
be embedded in $S^{\left|x\right|\cdot\lambda}\left(A\right)$ ---
choose for each $p\in S^{x}\left(A\right)$ some $b_{p}\in\mathfrak{C}$
realizing it; and send each $\left\langle p_{i}\right\rangle _{i<\lambda}$
to $\tp\left(\left\langle b_{p_{i}}\right\rangle _{i<\lambda}/A\right)$.
Obviously, this is injective.

For $p\in S_{2^{\lambda}}^{x}\left(A\right)$, let $\overline{p}\in\left(S^{x}\left(A\right)\right)^{\lambda}$
be the corresponding sequence (as in Example \ref{ex:set-algebra-types})
and take its image $q\in S^{\left|x\right|\cdot\lambda}\left(A\right)$
under this embedding. Then we define $f\left(p\right)=q$.

Assume $p_{1},p_{2}\in S_{2^{\lambda}}^{x}\left(A\right)$ and let
$q_{1}=f\left(p_{1}\right),q_{2}=f\left(p_{2}\right)$. If $q_{1},q_{2}$
are elementary conjugates as witnessed by $\pi$, so are $\overline{p_{1}},\overline{p_{2}}$
(in each coordinate) and thus also $p_{1},p_{2}$. 

Indeed for any $\varphi\left(x,a\right)$ and $i<\lambda$, 
\begin{align*}
p_{1}\left(\varphi\left(x,\pi^{-1}\left(a\right)\right)\right)_{i} & =1\Longleftrightarrow\mathfrak{C}\vDash\varphi\left(b_{\left(p_{1}\right)_{i}},\pi^{-1}\left(a\right)\right)\Longleftrightarrow\\
 & \mathfrak{C}\vDash\varphi\left(b_{\left(p_{2}\right)_{i}},a\right)\Longleftrightarrow\left(p_{2}\left(\varphi\left(x,a\right)\right)\right)_{i}=1.
\end{align*}

(where for any formula $\varphi$, $p\left(\varphi\right)\in2^{\lambda}$,
so $p\left(\varphi\right)$ is a function from $\lambda$ to 2 and
$p\left(\varphi\right)_{i}$ is the image of $i$.)

Thus $p_{1}\left(\varphi\left(x,\pi^{-1}\left(a\right)\right)\right)=p_{2}\left(\varphi\left(x,a\right)\right)\Rightarrow p_{2}=\pi\ast p_{1}$.
\end{proof}
As an immediate corollary we the following:
\begin{cor}
\label{cor:number-of-set-algebra-types}The number of types in $S_{2^{\lambda}}^{x}\left(A\right)$
up to elementary conjugation is at most the number of types in $S^{\lambda\left|x\right|}\left(A\right)$
up to conjugation.
\end{cor}

\begin{defn}
Fix some complete Boolean algebra $\Bb$ and regular cardinal $\kappa$.
$\mathcal{B}$ is $\kappa$-c.c if there is no antichain (a set of
pairwise disjoint elements from $\Bb^{+}$) in $\Bb$ of size $\kappa$. 
\end{defn}

\begin{defn}
Suppose $p\in S_{\Bb}^{x}\left(A\right)$, and $\varphi\left(x,y\right)$
is some formula. Then $p\upharpoonright\varphi$, or $p|_{\varphi}$,
is the restriction of $p$ to the definable sets of the form $\varphi\left(x,a\right)$
for $a$ some tuple (in the length of $y$) from $A$. 
\end{defn}

\begin{prop}
\label{prop:image not too big}Assume $\varphi\left(x,y\right)$ is
$NIP$ and $\mathcal{B}$ is $\kappa$-c.c. Then for any $p\in S_{\mathcal{B}}^{x}\left(A\right)$,
the image of $p|_{\varphi}$ has cardinality $\leq2^{<\kappa}$.

In particular if $T$ is $NIP$ this holds for every $\varphi$.
\end{prop}

\begin{proof}
Recall that a subset $X$ of $\Bb$ is \emph{independent} if every
nontrivial finite product from it is non-empty: for every $\mathfrak{a}_{1},\ldots,\mathfrak{a}_{n}$
and $\mathfrak{b}_{1},\ldots,\mathfrak{b}_{m}$ in $X$ such that
$\mathfrak{a}_{i}\neq\mathfrak{b}_{j}$ for all $i,j$, the product
$\mathfrak{a}_{1}\cdot\ldots\cdot\mathfrak{a}_{n}\cdot-\mathfrak{b}_{1}\cdot\ldots\cdot-\mathfrak{b}_{m}$
is not $0$. By \cite[Theorem 10.1]{MR593014,MR991565}, if $\lambda$
is a cardinal such that $\lambda$ is regular and $\mu^{<\kappa}<\lambda$
for all $\mu<\lambda$, then every subset $X\subseteq\Bb$ of cardinality
$\lambda$ has an independent subset $Y\subseteq X$ of cardinality
$\lambda$.

Let $\lambda=\left(2^{<\kappa}\right)^{+}$. It is easy to see that
for all $\mu<\lambda$, $\mu^{<\kappa}<\lambda$ --- since $\kappa$
is regular, for every $\beta<\kappa$ every function from $\left|\beta\right|$
to $\sup\left\{ 2^{\left|i\right|}\right\} _{i<\kappa}=2^{<\kappa}$
is contained in some $2^{\left|\alpha\right|}$ for $\alpha<\kappa$.

Thus 
\begin{align*}
\left(2^{<\kappa}\right)^{\left|\beta\right|} & \leq\stackrel[\alpha<\kappa]{}{\sum}\left|\left(2^{\left|\alpha\right|}\right)^{\left|\beta\right|}\right|=\stackrel[\alpha<\kappa]{}{\sum}2^{\left|\alpha\right|\left|\beta\right|}\leq\kappa\cdot\sup\left\{ 2^{\left|\alpha\right|\left|\beta\right|}\mid\alpha<\kappa\right\} =\\
 & \kappa\cdot\sup\left\{ 2^{\left|\alpha\right|}\mid\alpha<\kappa\right\} =2^{<\kappa},
\end{align*}

so 
\[
\mu^{<\kappa}=sup\left\{ \mu^{\left|\beta\right|}\right\} _{\beta<\kappa}\leq sup\left\{ \left(2^{<\kappa}\right)^{\left|\beta\right|}\right\} _{\beta<\kappa}\leq2^{<\kappa}<\lambda
\]

Since $T$ is NIP, $\varphi\left(x,y\right)$ has dual VC-dimension
$n<\omega$: for any $\left\langle a_{i}\right\rangle _{i<n+1}\in A^{\left|y\right|}$,
$\left\{ \varphi\left(b,y\right)\right\} _{b\in\mathfrak{C}^{\left|x\right|}}$
does not shatter $\left\langle a_{i}\right\rangle _{i<n+1}$.

Thus for any such sequence $\left\langle a_{i}\right\rangle _{i<n+1}$
exists $I\subseteq n+1$ such that 
\[
\stackrel[i\in I]{}{\bigwedge}\varphi\left(x,a_{i}\right)\wedge\stackrel[i\in\left(n+1\right)\setminus I]{}{\bigwedge}\neg\varphi\left(x,a_{i}\right)
\]

is inconsistent.

We get 
\[
\stackrel[i\in I]{}{\prod}p\left(\varphi\left(x,a_{i}\right)\right)\cdot\stackrel[i\in\left(n+1\right)\setminus I]{}{\prod}-p\left(\varphi\left(x,a_{i}\right)\right)=p\left(\stackrel[i\in I]{}{\bigwedge}\varphi\left(x,a_{i}\right)\wedge\stackrel[i\in\left(n+1\right)\setminus I]{}{\bigwedge}\neg\varphi\left(x,a_{i}\right)\right)=0,
\]

so $Im\left(p|_{\varphi}\right)$ has no independent set of size $n+1$,
let alone $\lambda$ --- thus $\left|Im\left(p|_{\varphi}\right)\right|<\lambda\Rightarrow\left|Im\left(p|_{\varphi}\right)\right|\leq2^{<\kappa}$.

\end{proof}
\begin{rem}
If $\varphi$ has IP and $\mathcal{B}$ is complete, then there is
a $\mathcal{B}$-type $p$ such that $Im\left(p|_{\varphi}\right)=\mathcal{B}$.
Indeed by compactness there is a sequence of parameters $\left\langle b_{\mathfrak{a}}\mid\mathfrak{a}\in\mathcal{B}\right\rangle $
such that $\left\langle \varphi\left(x,b_{\mathfrak{a}}\right)\mid\mathfrak{a}\in\mathcal{B}\right\rangle $
is independent as a subset of $L_{x}\left(\left\{ b_{\mathfrak{a}}\right\} _{\mathfrak{a}\in\mathcal{B}}\right)$,
and $f:\left\{ \varphi\left(x,b_{\mathfrak{a}}\right)\right\} _{\mathfrak{a}\in\mathcal{B}}\rightarrow\mathcal{B}$
defined as $f\left(\varphi\left(x,b_{\mathfrak{a}}\right)\right)=\mathfrak{a}$
can be extended to a $\mathcal{B}$-type by Fact \ref{fact:function-extension},
since there are no (non-trivial) Boolean relations on the domain.
\end{rem}

Now assume NIP and that $\mathcal{B}$ has $\kappa$-c.c.. Let ${\cal A}$
be the set of subalgebras $B$ of $\Bb$ of size at most $\mu=2^{<\kappa}+\left|x\right|+\left|T\right|$.
Proposition \ref{prop:image not too big} says that for $p\in\BtypesA$,
the algebra $B_{p}$ which is the image of $p$, is in ${\cal A}$.
For each $B\in{\cal A}$, choose some enumeration $\sequence{\mathfrak{a}_{B,i}}{i<\mu}$
(maybe with repetitions) of $B^{+}$. Let us say that $B_{1}$ is
conjugate to $B_{2}$ if there is a (unique) isomorphism $\sigma:B_{1}\to B_{2}$
taking $\mathfrak{a}_{B_{1},i}$ to $\mathfrak{a}_{B_{2},i}$.

Given $B\in{\cal A}$, for each $i<\mu$, choose an ultrafilter $D_{B,i}$
(of $B$) which contains $\mathfrak{a}_{B,i}$ (equivalently a homomorphism
from $B$ to $2$ such that $D_{B,i}\left(\mathfrak{a}_{B,i}\right)=1$,
see \cite[Proposition 2.15]{MR991565}). We ask that if $B_{1}$ and
$B_{2}$ are conjugates, say via $\sigma$, then $\sigma\left(D_{B_{1},i}\right)=D_{B_{2},i}$
--- or in the language of homomorphisms, $D_{B_{1},i}=D_{B_{2},i}\circ\sigma$.
To achieve this we choose representatives for the conjugacy classes
and a sequence of ultrafilters for each representative and construct
the others from it. Note that if $B_{1}$ is the chosen representative,
and $\sigma_{1}:B_{3}\rightarrow B_{2}$ and $\sigma_{2}:B_{2}\rightarrow B_{1}$
witness the conjugacy then $\sigma_{2}\circ\sigma_{1}:B_{3}\rightarrow B_{1}$
witnesses that $B_{1}$ and $B_{3}$ are conjugates. Thus $D_{B_{3},i}=D_{B_{1},i}\circ\left(\sigma_{2}\circ\sigma_{1}\right)=\left(D_{B_{1},i}\circ\sigma_{2}\right)\circ\sigma_{1}=D_{B_{2},i}\circ\sigma_{1}$.

Let $D_{B}=\left\langle D_{B,i}\right\rangle _{i<\mu}:B\rightarrow2^{\mu}$,
the product homomophism.

For $i<\mu$ and $p\in\BtypesA$, let $\widetilde{q}_{p}\in S_{2^{\mu}}^{x}\left(A\right)$
be the $2^{\mu}$-type $D_{B_{p}}\circ p$. Finally, let $q_{p}=f\left(\widetilde{q}_{p}\right)$
in $S^{\sequence{x_{i}}{i<\mu}}\left(A\right)$ where $f:S_{2^{\mu}}^{x}\left(A\right)\rightarrow S^{\sequence{x_{i}}{i<\mu}}\left(A\right)$
is as in Lemma \ref{lem:counting-set-algebra-types}.
\begin{prop}
\label{prop:attaching a real type to a homomorphism}(Assuming $\mathcal{B}$
has $\kappa$-c.c. and $T$ has $NIP$) Suppose $p_{1},p_{2}\in\BtypesA$
and $B=B_{p_{1}}=B_{p_{2}}$ and that $\widetilde{q}_{p_{1}}=\widetilde{q}_{p_{2}}$.
Then $p_{1}=p_{2}$. 
\end{prop}

\begin{proof}
If not, then for some $\varphi\left(x,a\right)$, $p_{1}\left(\varphi\left(x,a\right)\right)\neq p_{2}\left(\varphi\left(x,a\right)\right)$. 

WLOG $p_{1}\left(\varphi\left(x,a\right)\right)\nleq p_{2}\left(\varphi\left(x,a\right)\right)$,
and let $\mathfrak{b}=p_{1}\left(\varphi\left(x,a\right)\right)-p_{2}\left(\varphi\left(x,a\right)\right)$,
so $\mathfrak{b}\in B^{+}$. For some $i<\mu$, $\mathfrak{b}=\mathfrak{a}_{B,i}$.
It follows that $p_{1}\left(\varphi\left(x,a\right)\right)-p_{2}\left(\varphi\left(x,a\right)\right)\in D_{B,i}$,
hence $\widetilde{q}_{p_{1}}\left(\varphi\left(x,a\right)\right)_{i}=1$
and $\widetilde{q}_{p_{2}}\left(\varphi\left(x,a\right)\right)_{i}=0$
that is $\widetilde{q}_{p_{1}}\neq\widetilde{q}_{p_{2}}$.
\end{proof}
\begin{cor}
\label{cor:elementary conjugation}(Assuming $\mathcal{B}$ has $\kappa$-c.c.
and $T$ has $NIP$) Suppose $p_{1},p_{2}\in\BtypesA$ and $B=B_{p_{1}}=B_{p_{2}}$.
If $q_{p_{1}}$ and $q_{p_{2}}$ are conjugates over $A$, then $p_{1}$
and $p_{2}$ are elementarily conjugates over $A$. 
\end{cor}

\begin{proof}
Suppose that $\pi*q_{p_{1}}=q_{p_{2}}$ for an elementary permutation
$\pi:A\to A$. Then also, by Lemma \ref{lem:counting-set-algebra-types}
(to be precise, by the choice of $q_{p_{1}},q_{p_{2}}$), $\pi*\widetilde{q}_{p_{1}}=\widetilde{q}_{p_{2}}$.
Note that $\pi*\widetilde{q}_{p_{1}}=\widetilde{q}_{\pi*p_{1}}$.

But $B_{\pi\ast p_{1}}=B_{p_{1}}$ (since $\pi$ just permutes the
domain), thus we can apply Proposition \ref{prop:attaching a real type to a homomorphism}
to $\pi\ast p_{1},p_{2}$. 
\end{proof}
\begin{cor}
\label{cor:counting B types}(Assuming $\mathcal{B}$ has $\kappa$-c.c.
and $T$ has $NIP$) If $p_{1},p_{2}\in\BtypesA$, $B_{p_{1}}$ and
$B_{p_{2}}$ are conjugates and $q_{p_{1}}$ and $q_{p_{2}}$ are
conjugates over $A$, then $p_{1}$ and $p_{2}$ are conjugates over
$A$. 
\end{cor}

\begin{proof}
Suppose $\sigma:B_{p_{1}}\to B_{p_{2}}$ takes $\mathfrak{a}_{B_{p_{1}},i}$
to $\mathfrak{a}_{B_{p_{2}},i}$. Then $\sigma\circ p_{1}$ and $p_{2}$
satisfy the condition of Corollary \ref{cor:elementary conjugation}:
note that $B_{\sigma\circ p_{1}}=Im\left(\sigma\circ p_{1}\right)=\sigma\left(Im\left(p_{1}\right)\right)=\sigma\left(B_{p_{1}}\right)=B_{p_{2}}$.

Further, by the way we chose the $D_{B}$'s, $D_{B_{p_{1}}}=D_{B_{p_{2}}}\circ\sigma$
thus
\[
\widetilde{q}_{p_{1}}=D_{B_{p_{1}}}\circ p_{1}=D_{B_{p_{2}}}\circ\sigma\circ p_{1}=D_{B_{\sigma\circ p_{1}}}\circ\sigma\circ p_{1}=\widetilde{q}_{\sigma\circ p_{1}},
\]

and thus $q_{\sigma\circ p_{1}}=q_{p_{1}}$ and $q_{p_{2}}$ are conjugates.

So for some elementary permutation $\pi:A\to A$, $\pi*\left(\sigma\circ p_{1}\right)=p_{2}$

\end{proof}
\begin{cor}
\label{cor:counting}Assume $\mathcal{B}$ has $\kappa$-c.c. and
$T$ has $NIP$, and let $\mu=2^{<\kappa}+\left|T\right|+\left|x\right|$.

The number of types in $\BtypesA$ up to conjugation is bounded by
the number of types in $S^{\left|x\right|\cdot\mu}\left(A\right)$
(equivalently $S^{\mu}\left(A\right)$) up to conjugation + $2^{\mu}$.
Hence, if $\lambda,\varkappa$ are cardinals and $\alpha$ and ordinal
such that $\lambda=\lambda^{<\lambda}=\aleph_{\alpha}=\varkappa+\alpha\geq\varkappa\geq\beth_{\omega}+\mu^{+}$,
and if $A=M$ is a saturated model of size $\lambda$, this number
is bounded by $2^{<\varkappa}+\left|\alpha\right|^{\mu}+2^{\mu}=2^{<\varkappa}+\left|\alpha\right|^{\mu}$.
\end{cor}

\begin{proof}
Given $p$, map it to the pair consisting of the conjugacy class of
$q_{p}$ and the quantifier free type of the algebra $B_{p}$, enumerated
by $\sequence{\mathfrak{a}_{B_{p},i}}{i<\mu}$ (whose number is bounded
by $2^{\mu}$, since the language of Boolean algebras is finite).
By Corollary \ref{cor:counting B types}, the preimage of an element
under this map is contained in a single conjugacy class. The second
part of the statement follows immediately from Theorem 5.21 in \cite{Sh950},
since $\mu\geq\left|T\right|$. 
\end{proof}

\subsubsection{Counting types up to image conjugation}

One may wonder whether we can get a meaningful bound for $\left|S_{\mathcal{B}}^{x}\left(A\right)\right|$
(without taking elementary conjugation into consideration). For example,
the universal property of the product gives us a simple equation:
$\left|S_{2^{\mu}}^{x}\left(A\right)\right|=\left|S^{x}\left(A\right)\right|^{\mu}$. 

Note that if $\sigma:B\xrightarrow{\sim}B'$ is an isomorphism between
two different subalgebras of $\mathcal{B}$ and $p\in S_{\mathcal{B}}^{x}\left(A\right)$
is such that $Im\left(p\right)=B$, then $\sigma\circ p\in S_{\mathcal{B}}^{x}\left(A\right)$
is different from $p$. This means that if $\mathcal{B}$ has many
copies of small subalgebras then we necessarily have at least as many
$\mathcal{B}$ types. Thus if we want to give a bound that is independent
of the size of $\mathcal{B}$, we must restrict ourselves to counting
up to image conjugation. Note that $2$ has only a single subalgebra
(itself) and a single partial isomorphism (the identity), so for 2-types
(i.e., classical types) counting types up to image conjugation is
the same as just counting types.
\begin{cor}
Assume $\mathcal{B}$ has $\kappa$-c.c and $T$ has $NIP$, and assume
$\left|A\right|\geq\aleph_{0}$. We have the following bounds on $\lambda$,
the number of types in $S_{\mathcal{B}}^{x}\left(A\right)$ up to
image conjugation, where $\mu=2^{<\kappa}+\left|T\right|+\left|x\right|$.

If $T$ is stable, $\lambda\leq\left|A\right|^{\mu}$; 

If $T$ is NIP, $\lambda\leq\left(ded\left|A\right|\right)^{\mu}$,
where $ded\,\theta$ is the supremum on the number of cuts on a linearly
ordered set of cardinality $\leq\theta$.

If $T$ has IP and $x$ is finite then $\lambda$ could be maximal,
i.e., $\sup\set{\left|S_{\Bb}^{x}\left(A\right)\right|}{\left|A\right|\leq\varkappa}=\left|\Bb\right|^{\varkappa}$
for all $\varkappa\geq\left|\mathcal{B}\right|+\left|T\right|$. 
\end{cor}

\begin{proof}
The proof of Corollary \ref{cor:counting B types} shows also that
if $p_{1},p_{2}\in\BtypesA$, $B_{p_{1}}$ and $B_{p_{2}}$ are conjugates
and $q_{p_{1}}=q_{p_{2}}$, then $p_{1}$ and $p_{2}$ are image conjugates. 

We conclude that the number of $\mathcal{B}$ types over $A$ up to
image conjugation is at most $\left|S^{\mu}\left(A\right)\right|$.

Let $\phi\left(x,y\right)$ some partitioned formula, and denote by
$S_{\phi}\left(A\right)$ the set of $\phi$-types, that is, a maximal
consistent set of formulas of the form $\phi\left(x,b\right)$ or
$\neg\phi\left(x,b\right)$ where $b$ is a $y$-tuple in $A$.

For $\theta=\sup\left\{ \left|S_{\phi}\left(A\right)\right|\mid\phi\left(x,y\right)\right\} $,
$\left|S^{\mu}\left(A\right)\right|\leq\theta{}^{\mu}$, since the
number of formulas in $\mu$ variables is at most $\mu+\left|T\right|=\mu$,
and the map $p\mapsto\left\langle p|_{\phi}\right\rangle _{\phi}$
is injective.

According to \cite[Proposition 2.69]{pierrebook}, if $T$ has NIP
then $\theta\leq ded\left|A\right|$.

Further, by the preceding remarks there, if $T$ is stable then $\theta\leq\left|A\right|$
thus $\lambda\leq\left|A\right|^{\mu}$ for stable $T$.

Assume that $T$ has IP. Note that $2^{\varkappa}\leq\left|B\right|^{\varkappa}\leq\varkappa^{\varkappa}=2^{\varkappa}$.
If $\phi\left(x,y\right)$ has IP then there is some model $M\models T$
of size $\varkappa$ such that $\left|S_{\phi}\left(M\right)\right|=2^{\left|M\right|}$.
Note that the algebra $2$ is embedded in $\mathcal{B}$ so $S_{2}^{x}\left(M\right)$
embeds into $S_{\mathcal{B}}^{x}\left(M\right)$. Finally any two
of these types are not image conjugates since there is only one embedding
of $2$ to $\Bb$.
\end{proof}

\subsection{\label{subsec:Smooth-Boolean-types}Smooth Boolean types}

In this section we define the notion of a smooth Boolean type analogously
to a smooth Keisler measure: these are Boolean types which have a
unique extension to every larger parameter set. The main result of
this section is that every Boolean type in a NIP theory can be extended
to a smooth one (see Corollary \ref{cor:smooth-extension}). This
mirrors a similar result for Keisler measures (see \cite[Proposition 7.9]{pierrebook}),
which we recover later in subsection \ref{subsec:Smooth-Keisler-Measures}.
We also discuss the stronger notion of a realized Boolean type.
\begin{defn}
Let $A\subseteq B$, $p\in S_{\cal B}^{x}\left(A\right)$ and $q\in S_{\cal B}^{x}\left(B\right)$;
then $q$ \emph{extends} $p$ if it extends it as a function, that
is for any formula $\varphi\left(x,a\right)$ over $A$, $p\left(\varphi\left(x,a\right)\right)=q\left(\varphi\left(x,a\right)\right)$
(technically, the images of the equivalence classes are the same).

We say that $p$ is \emph{smooth} if for every such $B$ there exists
a unique $\mathcal{B}$-type $q$ over $B$ extending $p$.
\end{defn}

\begin{rem}
If $\mathcal{B}=2$ and $A=M$ is a model, a type is smooth iff it
is realized, that is equal to $\tp\left(a/M\right)$, for some tuple
$a$ in $M$.
\end{rem}

\begin{rem}
\label{rem:smooth-set-algebra}As we remarked in Example \ref{ex:set-algebra-types},
$p\in S_{2^{\kappa}}^{x}\left(A\right)$ is essentially equivalent
to a sequence $\left\langle p_{i}\right\rangle _{i<\kappa}$ of complete
types via $p\left(\varphi\right)_{i}=1\Longleftrightarrow\varphi\in p_{i}$;
it is obvious that $q$ extends $p$ iff $q_{i}$ extends $p_{i}$
for all $i$; and thus $p$ is smooth iff $p_{i}$ is smooth for each
$i$. 
\end{rem}

Assume $\left|x\right|<\omega$ (i.e., $x$ is a finite tuple). Then
this case gives rise naturally to the following definition:
\begin{defn}
A $\mathcal{B}$-type $p\in S_{\cal B}^{x}\left(M\right)$ over a
model $M$ is called \emph{realized}\textbf{ }if $\stackrel[a\in M]{}{\sum}p\left(x=a\right)$
exists and equals $1$. (Here and later, when we write $a\in M$ in
the sum, we mean $a\in M^{x}$.)
\end{defn}

\begin{rem}
\label{rem:set-algebra-realized}If $\mathcal{B}=2$, then $\stackrel[a\in M]{}{\sum}p\left(x=a\right)=1$
iff there exists $a\in M$ such that $p\left(x=a\right)=1$. Therefore
this definition agrees with the classical one for complete types.

If $\mathcal{B}=2^{\kappa}$, again by Example \ref{ex:set-algebra-types},
$p\in S_{\mathcal{B}}^{x}\left(M\right)$ corresponds to a sequence
of complete types and 
\begin{align*}
 & \stackrel[a\in M]{}{\sum}p\left(x=a\right)=1\Longleftrightarrow\forall i<\kappa\left(\stackrel[a\in M]{}{\sum}p\left(x=a\right)\right)_{i}=1\Longleftrightarrow\\
 & \forall i<\kappa\exists a\in M\left(x=a\right)\in p_{i}.
\end{align*}

That is $p$ is realized iff every $p_{i}$ is realized (in $M$).

Therefore, by Remark \ref{rem:smooth-set-algebra}, in this case $p$
is smooth iff every $p_{i}$ is smooth iff every $p_{i}$ is realized
iff $p$ is realized.
\end{rem}

One direction of the last remark works in general:
\begin{claim}
Assume $p\in S_{\mathcal{B}}^{x}\left(M\right)$ is realized. Then
$p$ is smooth.
\end{claim}

\begin{proof}
Let $\mathfrak{b}_{a}=p\left(x=a\right)\in\mathcal{B}$ for each $x$-tuple
$a$ in $M$. Assume $q\in S_{\mathcal{B}}^{x}\left(N\right)$ extends
$p$. Let $\varphi\left(x,c\right)\in L_{x}\left(N\right)$ be some
formula. 

Then by \cite[Lemma 1.33b]{MR991565}, since $\stackrel[a\in M]{}{\sum}\mathfrak{b}_{a}$
exists and equals $1$ by assumption,
\begin{align*}
q\left(\varphi\left(x,c\right)\right) & =q\left(\varphi\left(x,c\right)\right)\cdot\stackrel[a\in M]{}{\sum}\mathfrak{b}_{a}=\stackrel[a\in M]{}{\sum}q\left(\varphi\left(x,c\right)\right)\cdot\mathfrak{b}_{a}=\\
 & \stackrel[a\in M]{}{\sum}q\left(\varphi\left(x,c\right)\right)\cdot q\left(x=a\right)=\stackrel[a\in M]{}{\sum}q\left(x=a\wedge\varphi\left(x,c\right)\right).
\end{align*}

And in particular the RHS exists.

But for any $a\in M$, if $\mathfrak{C}\vDash\varphi\left(a,c\right)$
then $\left(x=a\wedge\varphi\left(x,c\right)\right)=\left(x=a\right)$
(as definable sets) thus $q\left(x=a\wedge\varphi\left(x,c\right)\right)=q\left(x=a\right)=\mathfrak{b}_{a}$,
while if $\mathfrak{C}\vDash\neg\varphi\left(a,c\right)$ then $\left(x=a\wedge\varphi\left(x,c\right)\right)=\bot$
thus $q\left(x=a\wedge\varphi\left(x,c\right)\right)=0$.

Thus we get necessarily (since the supremum never changes when adding
or removing $0$'s)
\[
q\left(\varphi\left(x,c\right)\right)=\stackrel[a\in\varphi\left(M,c\right)]{}{\sum}\mathfrak{b}_{a}.
\]

That is, we get $q$ is uniquely determined.
\end{proof}
One may naturally ask if every smooth type is realized. We start with
the following result:
\begin{claim}
\label{claim:smooth-implies-maximality} Assume $p\in S_{\mathcal{B}}^{x}\left(M\right)$
is smooth for $\mathcal{B}$ a complete Boolean algebra. Then ${\stackrel[a\in M]{}{\sum}p\left(x=a\right)}$
is \emph{maximal}: for any extension $q\in S_{\mathcal{B}}^{x}\left(N\right)$
of $p$, 
\[
\stackrel[a\in N]{}{\sum}q\left(x=a\right)=\stackrel[a\in M]{}{\sum}p\left(x=a\right).
\]
\end{claim}

\begin{proof}
Assume otherwise. Since for any $a\in M$, $q\left(x=a\right)=p\left(x=a\right)$,
we have in particular some $a\in N\setminus M$ be such that $q\left(x=a\right)>0$.

However, there is always a type $q'\in S_{\mathcal{B}}^{x}\left(N\right)$
extending $p$ such that $q'\left(x=a\right)=0$. Indeed, for any
consistent $\varphi\left(x,b\right)\in L_{x}\left(M\right)$ it cannot
be $\varphi\left(x,b\right)\rightarrow x=a$, as that would imply
that $a$ is definable over $M$; but $M\prec N$ thus its definable
closure is itself.

Thus by Fact \ref{fact:homomorphism-extension}, there exists $q'$
as required. 

This means that if $q\left(x=a\right)>0$, $q$ and $q'$ are distinct
extensions of $p$, thus $p$ is not smooth.
\end{proof}
The property in Claim \ref{claim:smooth-implies-maximality} has an
alternative formulation which is somewhat easier to reason about:
\begin{claim}
\label{claim:atoms-in-maximal}Assume $\mathcal{B}$ is complete.
A type $p\in S_{\mathcal{B}}^{x}\left(M\right)$ has the property
that $\stackrel[a\in M]{}{\sum}p\left(x=a\right)$ is not maximal
iff there exists a subalgebra $B'$ of $\mathcal{B}$ such that $Im\left(p\right)\subseteq B'$
and an atom $\mathfrak{a}\in B'$ such that $\mathfrak{a}\leq-\stackrel[a\in M]{}{\sum}p\left(x=a\right)$.
\end{claim}

\begin{proof}
Assume $q$ extends $p$ and $\stackrel[a\in N]{}{\sum}q\left(x=a\right)>\stackrel[a\in M]{}{\sum}p\left(x=a\right)$
and let $c\in N\setminus M$ such that $q\left(x=c\right)>0$. Let
$B'=Im\left(q\right)$ and $\mathfrak{a}=q\left(x=c\right)$. 

Then $\mathfrak{a}$ must be an atom in $B'$, since if $\mathfrak{b}\leq\mathfrak{a}$
and $\mathfrak{b}\in B'$, let $\varphi\left(x,b\right)$ such that
$q\left(\varphi\left(x,b\right)\right)=\mathfrak{b}$. If $\mathfrak{C}\vDash\varphi\left(c,b\right)$
then $\varphi\left(x,b\right)\wedge x=c$ is the same as $x=c$ and
thus $\mathfrak{b}=\mathfrak{b}\cdot\mathfrak{a}=q\left(\varphi\left(x,b\right)\wedge x=c\right)=q\left(x=c\right)=\mathfrak{a}$;
similarly if $\mathfrak{C}\nvDash\varphi\left(c,b\right)$ then $\varphi\left(x,b\right)\wedge x=c$
is $\bot$ thus we likewise get $\mathfrak{b}=0$. Finally since $\mathfrak{a}$
is disjoint from $p\left(x=a\right)$ for any $a\in M$, $\mathfrak{a}\leq-\stackrel[a\in M]{}{\sum}p\left(x=a\right)$.

On the other hand, let $B'$ and $\mathfrak{a}\in B'$ an atom. Let
$D_{\mathfrak{a}}:B'\rightarrow2$ be the principal ultrafilter generated
by $\mathfrak{a}$ represented as a homomorphism (i.e. $D_{\mathfrak{a}}\left(\mathfrak{b}\right)=1\Longleftrightarrow\mathfrak{a}\leq\mathfrak{b}$)
and let $p'=D_{\mathfrak{a}}\circ p$ which is a complete type over
$M$. $p'$ is not realized in $M$: as $\mathfrak{a}\leq-\stackrel[a\in M]{}{\sum}p\left(x=a\right)$,
$D_{\mathfrak{a}}\left(\stackrel[a\in M]{}{\sum}p\left(x=a\right)\right)=0$
thus $p'\left(x=a\right)=D_{\mathfrak{a}}\left(p\left(x=a\right)\right)\leq D_{\mathfrak{a}}\left(\stackrel[a\in M]{}{\sum}p\left(x=a\right)\right)=0$
for any $a$ in $M$.

Let $c$ realize $p'$ outside of $M$, and let $N$ containing $c$
and $M$. Then for any $\varphi\left(x,b\right)\in L_{x}\left(M\right)$,
if $\varphi\left(x,b\right)\rightarrow x=c$ then $\varphi\left(x,b\right)=\bot$
(since $c$ cannot be definable over $M$) thus 
\[
p\left(\varphi\left(x,b\right)\right)=0\leq\mathfrak{a};
\]

and if $x=c\rightarrow\varphi\left(x,b\right)$ then $\mathfrak{C}\vDash\varphi\left(c,b\right)$
thus 
\[
D_{\mathfrak{\mathfrak{a}}}\left(p\left(\varphi\left(x,b\right)\right)\right)=p'\left(\varphi\left(x,b\right)\right)=1\Rightarrow\mathfrak{a}\leq p\left(\varphi\left(x,b\right)\right).
\]

Thus by Fact \ref{fact:homomorphism-extension}, $p$ can be extended
to a type $q$ over $N$ which satisfies $q\left(x=c\right)=\mathfrak{a}$
thus 
\[
\stackrel[a\in N]{}{\sum}q\left(x=a\right)>\stackrel[a\in M]{}{\sum}p\left(x=a\right)
\]

as required.
\end{proof}
As a corollary of Claim \ref{claim:atoms-in-maximal} we get:
\begin{rem}
\label{rem:non-realized} If $p\in S_{\cal B}^{x}\left(M\right)$
is onto and $\mathcal{B}$ atomless and complete, then $\stackrel[a\in M]{}{\sum}p\left(x=a\right)$
is maximal and in fact $p\left(x=a\right)=0$ for all $a$. 

On the other hand if $\mathcal{B}$ is atomic (that is there is an
atom under every positive element) then $\stackrel[a\in M]{}{\sum}p\left(x=a\right)$
is maximal iff $\stackrel[a\in M]{}{\sum}p\left(x=a\right)=1$ (that
is iff $p$ is realized).
\end{rem}

\begin{example}
Let $L=\left\{ E_{B}\right\} _{B\in\mathcal{B}\left(\mathbb{R}\right)}$
(one unary predicate $E_{B}$ for every Borel set in $\mathbb{R}$),
$T=Th_{L}\left(\mathbb{R}\right)$ (with the obvious interpretations)
and $\mathcal{B}$ the algebra $\mathcal{B}\left(\mathbb{R}\right)/I$
where $I=\left\{ B\in\mathcal{B}\left(\mathbb{R}\right)\mid\mu\left(B\right)=0\right\} $
where $\mu$ is the Lebesgue measure; $\mathcal{B}$ is a $\sigma$-complete
and c.c.c. --- thus complete --- as well as atomless.

Then $T$ proves that every Boolean combination of the unary predicates
is equivalent to single unary predicate, and by a standard argument
eliminates quantifiers. Thus for any $M\subseteq\mathbb{R}$, $L_{x}\left(M\right)$
is isomorphic to $\mathcal{B}\left(\mathbb{R}\right)$ with $E_{B}\left(x\right)\mapsto B$
($x=a$ is equivalent to $E_{\left\{ a\right\} }\left(x\right)$).
Let $p:L_{x}\left(M\right)\rightarrow\mathbb{\mathcal{B}}$ be the
projection.

We get that for any $q:L_{x}\left(N\right)\to\cal B$ extending $p$,
$q$ is a surjection to an atomless Boolean algebra; therefore it
sends atoms to $0$, that is $q\left(x=a\right)=0$ for all $a\in N$.
Thus since $x=y$ is the only atomic formula involving both $x$ and
a parameter, and since $T$ eliminates quantifiers, $q$ is uniquely
determined.

Thus $p$ is smooth, but not realized by Remark \ref{rem:non-realized}.
\end{example}

\begin{defn}
Let $p\in S_{\mathcal{B}}^{x}\left(M\right)$ a Boolean type and $\varphi\left(x,y\right)$
a formula. We say that the image of $p$ with respect to $\varphi$
is maximal, or that $Im\left(p|_{\varphi\left(x,y\right)}\right)$
is maximal, if for any $N\supseteq M$ and for any $q\in S_{\mathcal{B}}^{x}\left(N\right)$
extending $p$ we have $\left\{ q\left(x,b\right)\mid b\in N\right\} =\left\{ p\left(x,a\right)\mid a\in M\right\} $. 

If the image of $p$ is maximal with respect to every $\varphi$ we
say that the image of $p$ is maximal.
\end{defn}

The following proposition gives us a way to extend types in a way
that maximizes their images, in the following precise sense:
\begin{prop}
\label{prop:maximal-image}Assume $\mathcal{B}$ is a Boolean algebra.
Let $p\in S_{\cal B}^{x}\left(M\right)$ a $\mathcal{B}$-type over
$M$. Then there exists $N\supseteq M$ and a type $q$ over $N$
extending $p$ such that the image of $q$ is maximal.
\end{prop}

\begin{proof}
Let $\left\langle \varphi_{i}\left(x,y\right)\mid i<\left|T\right|\right\rangle $
an enumeration of all partitioned formulas (recall that we are assuming
that $x$ is a finite tuple) and let $\left\langle \mathfrak{b}_{\alpha}\mid\alpha<\left|\mathcal{B}\right|\right\rangle $
be an enumeration of the elements of $\mathcal{B}$.

We construct recursively two increasing sequences with respect to
the lexicographic order on $\left(\left|T\right|+1\right)\times\left(\left|\mathcal{B}\right|+1\right)$:

1. An increasing sequence of models $\left\langle M_{i,\alpha}\mid\alpha\leq\left|\mathcal{B}\right|,i\leq\left|T\right|\right\rangle $.

2. An increasing (with respect to extension) sequence of $\mathcal{B}$-types
$\left\langle p_{i,\alpha}\mid\alpha\leq\left|\mathcal{B}\right|,i\leq\left|T\right|\right\rangle $
such that $p_{i,\alpha}\in S_{\mathcal{B}}^{x}\left(M_{i,\alpha}\right)$.

The construction is as follows:

For $\left(0,0\right)$, $M_{0,0}=M,p_{0,0}=p$

Fix $i$ and assume we have $M_{i,\alpha},p_{i,\alpha}$ for $\alpha<\left|\mathcal{B}\right|$;
if there exist $M'$ and $p'$ over $M'$ such that $M_{i,\alpha}\subseteq M'$,
$p'$ extends $p$ and $\mathfrak{b}_{\alpha}\in Im\left(p'|_{\varphi_{i}}\right)$,
let $M_{i,\alpha+1}=M'$, $p_{i,\alpha+1}=p'$; otherwise let $M_{i,\alpha+1}=M_{i,\alpha}$,
$p_{i,\alpha+1}=p_{i,\alpha}$. Assume we have $M_{i,\alpha}$, $p_{i,\alpha}$
for all $\alpha<\delta\leq\left|\mathcal{B}\right|$ a limit ordinal;
then define $M_{i,\delta}=\stackrel[\alpha<\delta]{}{\bigcup}M_{i,\alpha}$
and $p_{i,\delta}=\stackrel[\alpha<\delta]{}{\bigcup}p_{i,\alpha}$.
Note that since $\left(p_{i,\alpha}\right)_{\alpha<\delta}$ is a
chain of homomorphism this is a well defined homomorphism.

Assume we have $M_{i,\left|\mathcal{B}\right|},p_{i,\left|\mathcal{B}\right|}$
for $i<\left|T\right|$ and let $M_{i+1,0},p_{i+1,0}=M_{i,\left|\mathcal{B}\right|},p_{i,\left|\mathcal{B}\right|}$.
Finally assume we have $M_{i,\left|\mathcal{B}\right|},p_{i,\left|\mathcal{B}\right|}$
for all $i<j\leq\left|T\right|$ a limit ordinal. Then define $M_{j,0}=\stackrel[i<j]{}{\bigcup}M_{i,\left|\mathcal{B}\right|}$
and $p_{j,0}=\stackrel[i<j]{}{\bigcup}p_{i,\left|\mathcal{B}\right|}$.

Now let $N=M_{\left|T\right|,\left|\mathcal{B}\right|},q=p_{\left|T\right|,\left|\mathcal{B}\right|}$.
Then if for any $\mathfrak{b}_{\alpha}$ and $\varphi_{i}$ we have
that $\mathfrak{b}_{\alpha}\in Im\left(q'|_{\varphi_{i}}\right)$
for some extension $q'$ of $q$ then the same holds for $q$: any
extension of $q$ is also an extension of $p_{i,\alpha}$ and thus
by construction we have $\mathfrak{b}_{\alpha}\in Im\left(p_{i,\alpha+1}|_{\varphi_{i}}\right)\subseteq Im\left(q|_{\varphi_{i}}\right)$.
\end{proof}
\begin{rem}
\label{rem:cardinality of maximal image}If $\mathcal{B}$ has $\kappa$-c.c.
and $T$ is dependent, then by Proposition \ref{prop:image not too big}
in Proposition \ref{prop:maximal-image} we can take $N$ to be of
cardinality $\leq\left|M\right|+2^{<\kappa}+\left|T\right|$, by choosing
a preimage $\varphi\left(x,b\right)$ for every element in $Im\left(q\right)$
and restricting $q$ to a structure containing $M$ and each of the
$b$'s (which, from L\"{o}wenheim-Skolem, we can take to be no larger
than $\left|M\right|+2^{<\kappa}+\left|T\right|$).
\end{rem}

\begin{cor}
\label{cor:maximal-realization}Suppose $\mathcal{B}$ is complete.
Then every $\mathcal{B}$-type $p\in S_{\cal B}^{x}\left(M\right)$
can be extended to a $\mathcal{B}$-type $q\in S_{\mathcal{B}}^{x}\left(N\right)$
such that $\stackrel[a\in N]{}{\sum}q\left(x=a\right)$ is maximal.
\end{cor}

\begin{proof}
If the image of $p$ with respect to $x=y$ is maximal then $\stackrel[a\in M]{}{\sum}p\left(x=a\right)$
is maximal: indeed, assume $q$ over $N$ extends $p$. Then $\stackrel[a\in M]{}{\sum}p\left(x=a\right)<\stackrel[a\in N]{}{\sum}q\left(x=a\right)$
implies in particular that $Im\left(p|_{x=y}\right)\neq Im\left(q|_{x=y}\right)$.

We conclude from Proposition \ref{prop:maximal-image} that every
$\mathcal{B}$-type $p$ over $M$ has an extension $q$ over $N$
such that $\stackrel[a\in N]{}{\sum}q\left(x=a\right)$ is maximal.
\end{proof}
Corollary \ref{cor:maximal-realization} essentially reproves the
fact that if $\mathcal{B}=2^{\lambda}$ then every type can be extended
to a smooth type --- since in this case by Remark \ref{rem:non-realized},
$\stackrel[a]{}{\sum}q\left(x=a\right)$ is maximal iff $q$ is realized
and by Remark \ref{rem:set-algebra-realized} this happens iff $q$
is smooth. 

A similar approach can still be useful for any algebra. In the following
discussion we no longer need the notion of realized Boolean types,
and thus no longer assume that $x$ is a finite tuple. We start with
a useful claim:
\begin{claim}
\label{claim:non-conjugate-extensions}Assume $\mathcal{A},\mathcal{B}$
are Boolean algebras where $\cal B$ is complete, $\mathcal{A}'\subseteq\mathcal{A}$
a subalgebra, and $p:\mathcal{A}'\rightarrow\mathcal{B}$ and $p_{1},p_{2}:\mathcal{A}\rightarrow\mathcal{B}$
homomorphisms such that $p\subseteq p_{1},p_{2}$ and $p_{1}\neq p_{2}$.
Then there exist distinct extensions $q_{1},q_{2}:\mathcal{A}\rightarrow\mathcal{B}$
of $p$ and $\mathfrak{a}\in\cal A$ such that $q_{1}\left(\mathfrak{a}\right)<q_{2}\left(\mathfrak{a}\right)$
and $\sigma\circ q_{1}\neq q_{2}$ for any automorphism $\sigma$
of $\mathcal{B}$.
\end{claim}

\begin{proof}
Let $\mathfrak{a}\in\mathcal{A}\setminus\mathcal{A}'$ be such that
$p_{1}\left(\mathfrak{a}\right)\neq p_{2}\left(\mathfrak{a}\right)$. 

Let $\mathfrak{b}_{1}=\stackrel[\mathfrak{a}'\leq\mathfrak{a},\mathfrak{a}'\in\mathcal{A}']{}{\sum}p\left(\mathfrak{a}'\right)$
and $\mathfrak{b}_{2}=\stackrel[\mathfrak{a}'\geq\mathfrak{a},\mathfrak{a}'\in\mathcal{A}']{}{\prod}p\left(\mathfrak{a}'\right)$.
Then $\mathfrak{b}_{1}\leq p_{1}\left(\mathfrak{a}\right)\neq p_{2}\left(\mathfrak{a}\right)\leq\mathfrak{b}_{2}$
thus $\mathfrak{b}_{1}<\mathfrak{b}_{2}$.

By Fact \ref{fact:homomorphism-extension} there are extensions $q_{i}$
of $p$ such that $q_{i}\left(\mathfrak{a}\right)=\mathfrak{b}_{i}$
for $0\leq i<2$. Assume there is a homomorphism $\sigma:\mathcal{B}\rightarrow\mathcal{B}$
such that $\sigma\circ q_{1}=q_{2}$. 

Then for any $\mathfrak{b}\in Im\left(p\right)$, take $\mathfrak{a}'\in\mathcal{A}'$
such that $p\left(\mathfrak{a'}\right)=\mathfrak{b}$. Then $\sigma\left(\mathfrak{b}\right)=\sigma\left(p\left(\mathfrak{a}'\right)\right)=\sigma\left(q_{1}\left(\mathfrak{a}'\right)\right)=q_{2}\left(\mathfrak{a}'\right)=p\left(\mathfrak{a}'\right)=\mathfrak{b}$
thus $\sigma|_{Im\left(p\right)}=\id_{Im\left(p\right)}$.

Further $\sigma\left(\mathfrak{b}_{1}\right)=\sigma\left(q_{1}\left(\mathfrak{a}\right)\right)=q_{2}\left(\mathfrak{a}\right)=\mathfrak{b}_{2}$.
We conclude that $\mathfrak{b}_{2}=\sigma\left(\mathfrak{b}_{1}\right)\leq\sigma\left(\mathfrak{b}_{2}\right)$.
On the other hand, for any $\mathfrak{a}'\in\mathcal{A}'$ such that
$\mathfrak{a}'\geq\mathfrak{a}$, since by assumption $\mathfrak{b}_{2}\leq p\left(\mathfrak{a'}\right)$
we get $\sigma\left(\mathfrak{b}_{2}\right)\leq\sigma\left(p\left(\mathfrak{a}'\right)\right)=p\left(\mathfrak{a}'\right)$
thus $\sigma\left(\mathfrak{b}_{2}\right)\leq\stackrel[a'\geq a,a'\in A']{}{\prod}p\left(\mathfrak{a}'\right)=\mathfrak{b}_{2}$. 

We get $\sigma\left(\mathfrak{b}_{2}\right)=\mathfrak{b}_{2}=\sigma\left(\mathfrak{b}_{1}\right)$,
thus $\sigma$ is not injective and in particular, not an automorphism.
\end{proof}
\begin{prop}
\label{prop:maximal-is-smooth}Assume $\mathcal{B}$ is complete,
and assume that $N$ is a model of $T$ and $q$ is a $\mathcal{B}$-type
over $N$ that has the property in the conclusion of Proposition \ref{prop:maximal-image}:
it has maximal image with respect to every formula. Then for any formula
$\varphi\left(x,y\right)$, if there exist $N'\supseteq N$ and extensions
$q_{1},q_{2}\in S_{\mathcal{B}}^{x}\left(N'\right)$ of $q$  such
that $q_{1}|_{\varphi}\neq q_{2}|_{\varphi}$, then $\varphi$ is
independent.

In particular, if $T$ has NIP, such a $q$ is smooth.
\end{prop}

\begin{proof}
Assume otherwise, and let $n$ the such that $n$ is greater than
the dual VC-dimension of $\varphi\left(x,y\right)$ (see e.g., \cite[Lemma 6.3]{pierrebook}). 

Let $N'$ be an elementary extension of $N$, $q_{i}\in S_{\mathcal{B}}^{x}\left(N'\right)$
extending $q$ and $a\in N'$ such that $q_{i}\left(\varphi\left(x,a\right)\right)=\mathfrak{a}_{i}$.
By the previous claim we may assume $\mathfrak{a}_{1}<\mathfrak{a}_{2}$
and let $\mathfrak{b}=\mathfrak{a}_{2}-\mathfrak{a}_{1}>0$. Note
that by Fact \ref{fact:homomorphism-extension}, for any $\mathfrak{b}'\leq\mathfrak{b}$
there exists some $q'\in S_{\mathcal{B}}^{x}\left(N'\right)$ extending
$q$ such that $q'\left(\varphi\left(x,a\right)\right)=\mathfrak{a}_{1}+\mathfrak{b}'$,
thus by assumption (on $q$) there exists $a'\in N$ such that $q\left(\varphi\left(x,a'\right)\right)=\mathfrak{a}_{1}+\mathfrak{b}'$.

Assume first $\mathfrak{b}_{*}\leq\mathfrak{b}\leq\mathfrak{a}_{2},-\mathfrak{a}_{1}$
for some atom $\mathfrak{b}_{*}$ of $\mathcal{B}$. Let $x'\subseteq x$
be a finite tuple containing all variables from $x$ appearing in
$\varphi\left(x,y\right)$. Since $Im\left(q|_{x'=z}\right)$ is maximal,
we get by Claim \ref{claim:atoms-in-maximal} and the proof of Corollary
\ref{cor:maximal-realization} that $\mathfrak{b}_{*}\nleq-\stackrel[c\in N^{\left|x'\right|}]{}{\sum}q\left(x'=c\right)$;
but if $\mathfrak{b}_{*}\nleq q\left(x'=c\right)$ for all $c\in N^{\left|x'\right|}$
then since $\mathfrak{b}_{*}$ is an atom we get $\mathfrak{b}_{*}\leq-q\left(x'=c\right)$
for all $c\in N^{\left|x'\right|}$ that is 
\[
\mathfrak{b}_{*}\leq\stackrel[c\in N^{\left|x'\right|}]{}{\prod}-q\left(x'=c\right)=-\stackrel[c\in N^{\left|x'\right|}]{}{\sum}q\left(x'=c\right).
\]

Let then $c\in N^{\left|x'\right|}$ be such that $q\left(x'=c\right)\geq\mathfrak{b}_{*}$.
Let $D_{\mathfrak{b}_{*}}$ be the ultrafilter corresponding to $\mathfrak{b}_{*}$
represented as a homomorphism to $2$. Then $D_{\mathfrak{b}_{*}}\circ q|_{x'}\in S\left(A\right)$
is realized in $N$ (since $D_{\mathfrak{b}_{*}}\left(q\left(x'=c\right)\right)=1$).
But $D_{\mathfrak{b}_{*}}\circ q_{i}|_{x'}$ are extensions of $D_{\mathfrak{b}_{*}}\circ q|_{x'}$
and we have $D_{\mathfrak{b}_{*}}\left(q_{1}\left(\varphi\left(x,a\right)\right)\right)=D_{\mathfrak{b}_{*}}\left(\mathfrak{a}_{1}\right)=0$
and $D_{\mathfrak{b}_{*}}\left(q_{2}\left(\varphi\left(x,a\right)\right)\right)=D_{\mathfrak{b}_{*}}\left(\mathfrak{a}_{2}\right)=1$
so $D_{\mathfrak{b}_{*}}\circ q$ is not smooth --- contradiction.

Thus there is no atom of $\mathcal{B}$ under $\mathfrak{b}$, therefore
by trivial induction there exist disjoint $\mathfrak{b}_{\eta}>0$
for $\eta\in\left\{ \pm1\right\} ^{n}$ such that $\stackrel[\eta\in\left\{ \pm1\right\} ^{n}]{}{\sum}\mathfrak{b}_{\eta}=\mathfrak{b}$.
Let $\mathfrak{b}_{i}=\stackrel[\eta\in\left\{ \pm1\right\} ^{n},\eta\left(i\right)=1]{}{\sum}\mathfrak{b}_{\eta}\leq\mathfrak{b}$.
Then $\mathfrak{b}-\mathfrak{b}_{i}=\stackrel[\eta\in\left\{ \pm1\right\} ^{n},\eta\left(i\right)=-1]{}{\sum}\mathfrak{b}_{\eta}$
thus for any such $\eta$, $\mathfrak{b}\stackrel[i<n]{}{\prod}\eta\left(i\right)\mathfrak{b}_{i}=\mathfrak{b}_{\eta}>0$.

By assumption on $q$, for any $i$ there exists $a_{i}\in N^{\left|y\right|}$
such that $q\left(\varphi\left(x,a_{i}\right)\right)=\mathfrak{a}_{1}+\mathfrak{b}_{i}$. 

Let $\varphi_{\eta}\left(x,y_{0},...,y_{n-1}\right)=\stackrel[i<n]{}{\bigwedge}\varphi\left(x,y_{i}\right)^{\eta\left(i\right)}$.
Then 
\begin{align*}
q\left(\varphi_{\eta}\left(x,a_{0},...,a_{n-1}\right)\right) & =\prod_{i<n}\eta\left(i\right)\left(\mathfrak{a}_{1}+\mathfrak{b}_{i}\right)\geq\\
 & \mathfrak{b}\prod_{i<n}\eta\left(i\right)\left(\mathfrak{a}_{1}+\mathfrak{b}_{i}\right).
\end{align*}

Since for every $i<n$, $\mathfrak{b}\cdot\left(-\mathfrak{b}_{i}\right)=\mathfrak{b}-\mathfrak{b}_{i}\leq\mathfrak{b}\leq-\mathfrak{a}_{1}$,
we have $\mathfrak{b}\cdot-\left(\mathfrak{a}_{1}+\mathfrak{b}_{i}\right)=\left(-\mathfrak{a}_{1}\right)\cdot\mathfrak{b}\cdot\left(-\mathfrak{b}_{i}\right)=\mathfrak{b}-\mathfrak{b}_{i}$;
therefore $\mathfrak{b}\cdot\eta\left(i\right)\left(\mathfrak{a}_{1}+\mathfrak{b}_{i}\right)\geq\mathfrak{b}\cdot\eta\left(i\right)\mathfrak{b}_{i}$.

Thus 
\begin{align*}
q\left(\varphi_{\eta}\left(x,a_{0},...,a_{n-1}\right)\right) & \geq\mathfrak{b}\prod_{i<n}\eta\left(i\right)\left(\mathfrak{a}_{1}+\mathfrak{b}_{i}\right)\geq\\
 & \mathfrak{b}\prod_{i<n}\eta\left(i\right)\mathfrak{b}_{i}=\mathfrak{b}_{\eta}>0
\end{align*}

thus $\varphi_{\eta}\left(x,a_{0},...,a_{n-1}\right)\neq\bot$ thus
$N\vDash\exists x\varphi_{\eta}\left(x,a_{0},...,a_{n-1}\right)$
contradicting our choice of $n$. 
\end{proof}
By Proposition \ref{prop:maximal-image} and remark \ref{rem:cardinality of maximal image}
we get:
\begin{cor}
\label{cor:smooth-extension}If $\mathcal{B}$ is a complete Boolean
algebra, every $\mathcal{B}$-type in an NIP theory can be extended
to a smooth $\mathcal{B}$-type. In fact, if $\cal B$ is $\kappa$-c.c.,
then given a $\mathcal{B}$-type $p\in S_{\cal B}^{x}\left(M\right)$
we can find some $N\succ M$ and smooth $q\in S_{\cal B}^{x}\left(N\right)$
where $\left|N\right|\leq\left|M\right|+\left|x\right|+2^{<\kappa}+\left|T\right|$.
\end{cor}

\section{\label{sec:Relation-to-Keisler}Relation to Keisler Measures}

\subsection{\label{subsec:Connecting-Keisler-Measures}Connecting Keisler Measures
and Boolean Types}

In this subsection we recall the notion of a Keisler measure, and
attach to a Keisler measure a Boolean type in a canonical probability
algebra in a way that preserves the measure. This preserves many of
the measure's properties and will be used later to transfer results
from Boolean types to Keisler measures.
\begin{defn}
A \emph{Keisler measure} in $x$ over a set $A$ is a finitely additive
probability measure on $L_{x}\left(A\right)$. 

Two Keisler measures $\lambda,\lambda'$ in $x$ over $A$ are conjugates
if exists an elementary map $\pi:A\rightarrow A$ such that $\lambda\left(\varphi\left(x,\pi^{-1}\left(a\right)\right)\right)=\lambda'\left(\varphi\left(x,a\right)\right)$
for any formula $\varphi\left(x,a\right)$.
\end{defn}

\begin{defn}
\label{def:A-measure-algebra}A \emph{measure algebra} is a $\sigma$-complete
Boolean algebra $\mathcal{B}$ (not necessarily an algebra of sets)
equipped with a $\sigma$-additive measure that is positive on every
element other than $0_{\mathcal{\mathcal{B}}}$.

A \emph{probability algebra }is a measure algebra that assigns measure
$1$ to $1_{\mathcal{B}}$.
\end{defn}

\begin{example}
For every probability space, the algebra of measurable subsets up
to measure 0 is a probability algebra.
\end{example}

\begin{defn}
Let $\kappa$ an infinite cardinal. Then $\left(\mathcal{U}_{\kappa},\nu_{\kappa}\right)$
is the probability algebra of Borel subsets of $2^{\kappa}$ up to
measure $0$, with $\nu_{\kappa}$ the usual product measure (see
\cite[254J]{MR2462280}).
\end{defn}

\begin{rem}
Since every probability algebra has c.c.c., every supremum or infimum
is effectively countable. Thus in particular the $\sigma$-complete
subalgebra generated by some subset is the same as the complete subalgebra
generated by the same set.
\end{rem}

We will show that we can attach to a Keisler measure a $\mathcal{B}$-type
for a measure algebra $\mathcal{B}$.
\begin{rem}
\label{rem:generation-of-Keisler}Given a Keisler measure $\lambda$
over $A$, we can consider it as a measure on clopen sets of $S^{x}\left(A\right)$
and then extend it uniquely to a regular $\sigma$-additive measure
on the Borel sets of $S^{x}\left(A\right)$ (see \cite[416Qa]{FremlinMeasureTheoryVol4}). 

Let $\left(\Bb,\lambda\right)$ be the probability algebra of Borel
subsets of $S^{x}\left(A\right)$ up to $\lambda$ measure $0$ and
let $\psi$ be the projection from the algebra of Borel subsets onto
$\mathcal{B}$. Since the clopen sets are a basis for the topology,
the complete subalgebra of $\Bb$ generated by the clopen sets is
$\mathcal{B}$ itself. Note that there are at most $\left|T\right|+\left|A\right|+\left|x\right|$
clopen sets.
\end{rem}

\begin{fact}
\label{fact:embedding-measures}Let $\kappa$ an infinite cardinal.
If $\mathcal{B}$ is a probability algebra, and there is $B\subseteq\mathcal{B}$
such that $\left|B\right|\leq\kappa$ and the smallest ($\sigma$-)
complete subalgebra of $\mathcal{B}$ containing $B$ is $\mathcal{B}$
itself then there is a measure preserving homomorphism $f$ from $\Bb$
to $\left(\mathcal{U}_{\kappa},\nu_{\kappa}\right)$ (see \cite[Lemma 332N]{MR2459668},
and see also the proposition on page 126 and 331F there).

Further, every measure algebra homomorphism is an embedding.
\end{fact}

\begin{prop}
\label{prop:extending-partial-isomorphisms}Assume $\mathcal{A}\subseteq\mathcal{U}_{\kappa}$
is a complete subalgebra that can be completely generated by a set
$S$ such that $\left|S\right|<\kappa$. Assume further that $f:\mathcal{A}\rightarrow\mathcal{U}_{\kappa}$
is measure preserving (thus an embedding).

Then exists a measure preserving automorphism $\sigma$ of $\mathcal{U}_{\kappa}$
extending $f$.
\end{prop}

\begin{proof}
By following the proof for \cite[Theorem 331I]{MR2459668} ($\mathcal{U}_{\kappa}$
satisfies the requirements for the theorem by \cite[Theorem 331K]{MR2459668}),
we find that a recursive construction of an automorphism can start
from any partial isomorphism, as long as the domain of said partial
isomorphism is a complete subalgebra (in Fremlin's terminology, closed
subalgebra) that can be completely generated by less than $\kappa$
elements.
\end{proof}
\begin{rem}
\label{rem:extra-structure}When considering Boolean types to a measure
algebra $\left(\Bb,\lambda\right)$, we will adapt the definitions
of conjugate types (Definition \ref{def:generalized types}) and conjugate
subalgebras to this context, which means that partial isomorphisms
are now required to keep the extra structure, (that is, to preserve
the measure). 
\end{rem}

\begin{prop}
\label{prop:attaching a general type to a measure} Let $\kappa\geq\left|A\right|+\left|T\right|+\left|x\right|$
be some cardinal. There is an injection from the set of Kiesler measures
over a set $A$ to the set of $\left(\mathcal{U}_{\kappa},\nu_{\kappa}\right)$-types
over $A$. Further, this injection respects conjugation; that is,
if the images of $\lambda,\lambda'$ are conjugate with $\pi:A\rightarrow A$,
then so are $\lambda,\lambda'$.
\end{prop}

\begin{proof}
Let $f$ and $\psi$ be as in Fact \ref{fact:embedding-measures},
using Remark \ref{rem:generation-of-Keisler}.

Let $p_{\lambda}:L_{x}\left(A\right)\to\mathcal{U}_{\kappa}$ be $f\circ\psi|_{L_{x}\left(A\right)}$
(where $L_{x}\left(A\right)$ is thought of as the algebra of clopen
subsets of $S^{x}\left(A\right)$). Then, $\lambda\mapsto p_{\lambda}$
would be our injection.

By choice of $f$, $\lambda\left(\varphi\right)=\nu_{\kappa}\left(p_{\lambda}\left(\varphi\right)\right)$,
which means that this map is indeed injective. It follows that if
$p_{\lambda_{1}}$ and $p_{\lambda_{2}}$ are conjugate (as in Remark
\ref{rem:extra-structure}, i.e., as $\left(\mathcal{U}_{\kappa},\nu_{\kappa}\right)$-types)
then $\lambda_{1}$ and $\lambda_{2}$ are conjugate:

Suppose that $\pi:A\to A$ is an elementary map and that $\sigma:B_{p_{\lambda_{1}}}\to B_{p_{\lambda_{2}}}$
is a measure preserving isomorphism such that $\sigma\circ\left(\pi*p_{\lambda_{1}}\right)=p_{\lambda_{2}}$.
Then 
\begin{align*}
\pi*\lambda_{1}\left(\varphi\left(x,a\right)\right) & =\lambda_{1}\left(\varphi\left(x,\pi^{-1}\left(a\right)\right)\right)=\nu_{\kappa}\left(p_{\lambda_{1}}\left(\varphi\left(x,\pi^{-1}\left(a\right)\right)\right)\right)\\
 & =\nu_{\kappa}\left(\sigma\left(p_{\lambda_{1}}\left(\varphi\left(x,\pi^{-1}\left(a\right)\right)\right)\right)\right)=\nu_{\kappa}\left(p_{\lambda_{2}}\left(\varphi\left(x,a\right)\right)\right)\\
 & =\lambda_{2}\left(\varphi\left(x,a\right)\right).
\end{align*}
\end{proof}

\subsection{\label{subsec:Counting-Keisler-Measures}Counting Keisler Measures}

In this subsection we count the number of Keisler measures up to elementary
conjugation similarly to Subsection \ref{subsec:Counting-Boolean-types}.
\begin{rem}
\label{rem:counting-extra-structure}Corollary \ref{cor:counting B types}
and all proofs leading to it still work when we add structure to the
algebra, but now we have to take into account the number of isomorphism
types of the new structure, i.e. the measure. Recall that in Subsection
\ref{subsec:Counting-Boolean-types} we defined $\mu=2^{<\kappa}+\left|T\right|+\left|x\right|$
(for $\mathcal{B}$ $\kappa$-c.c.), the maximal cardinality of the
image of any single $\mathcal{B}$-type when $T$ is NIP. In the case
of a measure algebra, the number of possible isomorphism types is
thus $\leq2^{\mu}+\left(2^{\aleph_{0}}\right)^{\mu}=2^{\mu}$ (the
isomorphism type of an measure algebra $B$ consists of the quantifier
free type of $B$ as in the proof of Corollary \ref{cor:counting}
and the list of values $\left\{ \lambda\left(\mathfrak{a}\right)\right\} _{\mathfrak{a}\in B}$).
Finally, note that every probability algebra is c.c.c (i.e., $\omega_{1}$-c.c.),
so in this case we can take $\kappa=\omega_{1}$ and then $\mu=2^{\aleph_{0}}+\left|T\right|+\left|x\right|$.
\end{rem}

By applying Proposition \ref{prop:attaching a general type to a measure},
Remark \ref{rem:counting-extra-structure} and Corollary \ref{cor:counting}
we get that:
\begin{cor}
\label{cor: final corollary for measures}Assume $T$ has $NIP$.
The number of Keisler measures in $x$ up to conjugation over a set
$A$ is bounded by the number of types in $S^{\left|x\right|\cdot\mu}\left(A\right)$
up to conjugation + $2^{\mu}$ where $\mu=2^{\aleph_{0}}+\left|T\right|+\left|x\right|$.

Hence, if again $\lambda,\varkappa$ are cardinals and $\alpha$ an
ordinal such that $\lambda=\lambda^{<\lambda}=\aleph_{\alpha}=\varkappa+\alpha\geq\varkappa\geq\beth_{\omega}+\mu^{+}$,
and $A=M$ is a saturated model of size $\lambda$, this number is
bounded by $2^{<\varkappa}+\left|\alpha\right|^{\mu}$.
\end{cor}

However, Keisler measures have been studied extensively in the context
of NIP (see for example \cite[Chapter 7]{pierrebook}), and there
are results which give a better bound. Indeed:
\begin{prop}
\label{prop:keisler-measures-in-NIP} Assume $T$ has $NIP$. Then
there is an injection from the set of Keisler measures in $x$ over
a set $A$ to the set of $2^{\mu}$-types over A where $\mu=\left|T\right|+\left|x\right|$.

Further, if the images of $\lambda,\lambda'$ are elementarily conjugate,
$\lambda$ and $\lambda'$ are conjugates.
\end{prop}

\begin{proof}
Recall that given complete types $p_{0},...,p_{n-1}$, and a formula
$\phi\left(x,b\right)$, $Av\left(p_{0},..,p_{n-1};\phi\left(x,b\right)\right)$
is $\frac{\left|\left\{ i<n\mid\varphi\left(x,b\right)\in p_{i}\right\} \right|}{n}$.

Fix a Keisler measure $\lambda$ over $A$. By \cite[Proposition 7.11]{pierrebook},
(for $X_{1}=\top$) for any $m<\omega$ and partitioned formula $\phi\left(x,y\right)$
there exist $n_{m,\phi}<\omega$ and types $\left\langle p_{i}^{m,\phi}\left(x\right)\right\rangle _{i<n_{m,\phi}}\in S^{x}\left(A\right)$
such that for any $y$-tuple $b$, 
\[
\left|\lambda\left(\phi\left(x;b\right)\right)-Av\left(p_{0}^{m,\phi},...,p_{n_{m,\phi}-1}^{m,\phi};\phi\left(x,b\right)\right)\right|<\frac{1}{m}.
\]

Note that by \cite[Exercise 7.12]{pierrebook}, $n_{m,\phi}$ can
be chosen independently of $\lambda$.

Let $I=\left\{ m,\phi,i\mid m<\omega,\phi\in L_{x}\left(\emptyset\right),i<n_{\phi,m}\right\} $
(note that $\left|S\right|=\mu$). Choose $p_{\lambda}\in S_{2^{I}}^{x}\left(A\right)$
such that $p_{\lambda}\left(\phi\left(x,b\right)\right)_{m,\phi,i}=1$
iff $\phi\left(x,b\right)\in p_{i}^{m,\phi}$.

Then
\[
\stackrel[m\rightarrow\infty]{}{\lim}Av\left(\left(p_{\lambda}\right)_{m,\phi,0},...,\left(p_{\lambda}\right)_{m,\phi,n_{m,\phi}-1};\phi\left(x,b\right)\right)=\lambda\left(\phi\left(x,b\right)\right)
\]

And if $\pi*p_{\lambda_{1}}=p_{\lambda_{2}}$ we find that for any
$\phi\left(x,y\right)$ and $b$, 
\begin{align*}
\pi*\lambda_{1}\left(\phi\left(x,b\right)\right) & =\lambda_{1}\left(\phi\left(x,\pi^{-1}\left(b\right)\right)\right)=\stackrel[m\rightarrow\infty]{}{\lim}Av\left(\left(p_{\lambda_{1}}\right)_{m,\phi,0},...,\left(p_{\lambda_{1}}\right)_{m,\phi,n_{m,\phi}-1};\phi\left(x,\pi^{-1}\left(b\right)\right)\right)\\
 & \stackrel[m\rightarrow\infty]{}{\lim}Av\left(\left(p_{\lambda_{2}}\right)_{m,\phi,0},...,\left(p_{\lambda_{2}}\right)_{m,\phi,n_{m,\phi}-1};\phi\left(x,\pi^{-1}\left(b\right)\right)\right)=\lambda_{2}\left(\phi\left(x,b\right)\right).
\end{align*}
\end{proof}
Thus with Corollary \ref{cor:number-of-set-algebra-types} we get 
\begin{cor}
Assume $T$ has $NIP$. The number of Keisler measures up to conjugation
over a set $A$ is bounded by the number of types in $S^{\left|x\right|\cdot\mu}\left(A\right)$
up to conjugation where $\mu=\left|T\right|+\left|x\right|$ (with
the same explicit bound as in Corollary \ref{cor: final corollary for measures},
but replacing $\mu=2^{\aleph_{0}}+\left|T\right|+\left|x\right|$
with $\mu=\left|T\right|+\left|x\right|$). 
\end{cor}

\begin{rem}
This bound improves upon the one in Corollary \ref{cor: final corollary for measures},
since $\mu$ is now potentially smaller (for small $\left|T\right|$
and $\left|x\right|$).
\end{rem}

\subsection{\label{subsec:Smooth-Keisler-Measures}Smooth Keisler Measures}

Recall that a Keisler measure $\lambda$ is \emph{smooth} if it has
a unique extension to any set containing its domain. In this subsection
we show that Proposition \ref{prop:attaching a general type to a measure}
preserves smoothness, that is the Boolean type is smooth iff the measure
is. We use this to recover the fact that in an NIP theory measures
can be extended to smooth ones.
\begin{lem}
\label{lem:extend-embedding}Assume $\lambda_{i}$ is a Keisler measure
over $A_{i}$ for $i\in\left\{ 1,2\right\} $ such that $A_{1}\subseteq A_{2}$
and $\lambda_{2}$ extends $\lambda_{1}$; let $\kappa_{i}=\left|A_{i}\right|+\left|T\right|+\left|x\right|$
and let $f_{i}:\mathcal{B}_{i}\rightarrow\mathcal{U}_{\kappa_{i}^{+}}$
be an embedding of measure algebras, where $\mathcal{B}_{i}$ is the
algebra of Borel subsets of $S^{x}\left(A_{i}\right)$ up to $\lambda_{i}$-measure
$0$ (see Remark \ref{rem:generation-of-Keisler}).

Fix an embedding $\iota$ of $\mathcal{U}_{\kappa_{1}^{+}}$ in $\mathcal{U}_{\kappa_{2}^{+}}$
(one exists by Fact \ref{fact:embedding-measures}).

Then $\cal B_{1}$ embeds canonically into $\cal B_{2}$ and there
is a measure algebra homomorphism $g:\mathcal{B}_{2}\rightarrow\mathcal{U}_{\kappa_{2}^{+}}$
that extends $\iota\circ f_{1}$ and such that $\nu_{\kappa_{2}^{+}}\circ f_{2}=\nu_{\kappa_{2}^{+}}\circ g$.
\end{lem}

\begin{proof}
Note first that since $\lambda_{2}$ extends $\lambda_{1}$, $\mathcal{B}_{1}$
can be embedded into $\mathcal{B}_{2}$ naturally with the preimage
of the projection map from $S^{x}\left(A_{2}\right)$ to $S^{x}\left(A_{1}\right)$.

Note that $\left(\iota\circ f_{1}\right)\left(\mathcal{B}_{1}\right)$
is generated as a ($\sigma$-)complete algebra by $\kappa_{1}$ elements
(as the image of such an algebra) and that it is the complete subalgebra
of $\mathcal{U}_{\kappa_{2}^{+}}$, generated as a complete subalgebra
by the images of the clopen sets over $A_{1}$ (see \cite[Proposition 324L]{MR2459668}).

By Proposition \ref{prop:extending-partial-isomorphisms} and as $\lambda_{2}$
extends $\lambda_{1}$, there is an automorphism $\sigma$ of $\left(\mathcal{U}_{\kappa_{2}^{+}},\mathcal{\nu}_{\kappa_{2}^{+}}\right)$
extending $f_{2}\circ\left(\iota\circ f_{1}\right)^{-1}$.

But now we get $\sigma^{-1}\circ f_{2}|_{\mathcal{B}_{1}}=\iota\circ f_{1}|_{\mathcal{B}_{1}}$,
thus $g=\sigma^{-1}\circ f_{2}$ is as required.
\end{proof}
\begin{cor}
\label{cor:smooth-type-implies-smooth-measure}If $p\in S_{\mathcal{U}_{\kappa^{+}}}^{x}\left(M\right)$
is smooth, then so is $\nu_{\kappa^{+}}\circ p$ for $\kappa\geq\left|M\right|+\left|T\right|+\left|x\right|$.
\end{cor}

\begin{proof}
Assume $\lambda=\nu_{\kappa^{+}}\circ p$ is not smooth. So there
is $N\supseteq M$ and distinct $\lambda_{1},\lambda_{2}$ over $N$
extending $\lambda$, and hence there is one of cardinality at most
$\left|M\right|$ by L\"{o}wenheim Skolem (restrict $\lambda_{1},\lambda_{2}$
to a smaller model containing $M$ and some $b$, for which exists
$\varphi\left(x,b\right)$ such that $\lambda_{1}\left(\varphi\left(x,b\right)\right)\neq\lambda_{2}\left(\varphi\left(x,b\right)\right)$).

Let $\mathcal{B}_{0}$ be the measure algebra of Borel sets in $S^{x}\left(M\right)$
up to $\lambda$-measure $0$ and $I$ the ideal of sets of $\lambda$-measure
$0$ in the algebra of Borel subsets of $S^{x}\left(M\right)$ (see
Remark \ref{rem:generation-of-Keisler}). Then since $\nu_{\kappa^{+}}\left(\mathfrak{a}\right)=0\Longleftrightarrow\mathfrak{a}=0$,
$I\cap L_{x}\left(M\right)=\ker\left(p\right)$ thus $L_{x}\left(M\right)/\ker\left(p\right)$
is naturally embedded in $\mathcal{B}_{0}$.

Write $p=\widetilde{p}\circ\pi$ for $\pi:L_{x}\left(M\right)\rightarrow L_{x}\left(M\right)/\ker\left(p\right)$
the projection. 

Then $\widetilde{p}:L_{x}\left(M\right)/\ker\left(p\right)\rightarrow\mathcal{U}_{\kappa^{+}}$
is a measure preserving function from a subalgebra of $\mathcal{B}_{0}$,
and the complete (in Fremlin's terminology, order-closed) subalgebra
of $\mathcal{B}_{0}$ generated by $L_{x}\left(M\right)/\ker\left(p\right)$
is $\mathcal{B}_{0}$ itself. Thus by \cite[Proposition 324O and Proposition 323J]{MR2459668}
$\widetilde{p}$ has a unique extension to a measure preserving $\overline{p}:\mathcal{B}_{0}\rightarrow\mathcal{U}_{\kappa^{+}}$,
that is $\lambda=\nu_{\kappa^{+}}\circ\overline{p}$ (here we treat
$\lambda$ as a measure on $S^{x}\left(M\right)$). 

For each $i\in\left\{ 1,2\right\} $, we do the following. Let $\mathcal{B}_{i}$
be the measure algebra of Borel subsets of $S^{x}\left(N\right)$
up to $\lambda_{i}$-measure $0$, which naturally embeds $\mathcal{B}_{0}$
as a subalegbra via the preimage map of the projections $S^{x}\left(N\right)\to S^{x}\left(M\right)$
(this uses the fact that $\lambda_{i}$ extends $\lambda_{0}$). Take
an embedding $f_{i}:\mathcal{B}_{i}\rightarrow\mathcal{U}_{\kappa^{+}}$
for $\lambda_{i}$ guaranteed by Fact \ref{fact:embedding-measures},
and using Lemma \ref{lem:extend-embedding} we can find a measure
preserving $g_{i}:\mathcal{B}_{i}\rightarrow\mathcal{U}_{\kappa^{+}}$
extending $\overline{p}$ such that $\lambda_{i}=\nu_{\kappa^{+}}\circ g_{i}$.

Let $\pi_{i}:L_{x}\left(N\right)\rightarrow\mathcal{B}_{i}$ be the
projection, note that it extends $\pi$. Then for $p_{i}=g_{i}\circ\pi_{i}\in S_{\mathcal{U}_{\kappa^{+}}}^{x}\left(N\right)$
we find $\lambda_{i}=\nu_{\kappa^{+}}\circ p_{i}$ (when we consider
$\lambda_{i}$ as Keisler measures), and since $\pi_{i}$ extends
$\pi$ and $g_{i}$ extend $\overline{p}$ (thus $\widetilde{p}$)
we find $p_{1},p_{2}$ extend $p$, but they are distinct since $\lambda_{1}$
and $\lambda_{2}$ are distinct.

We conclude that $p$ is not smooth.
\end{proof}
On the other hand:
\begin{prop}
\label{cor:smooth-measure-implies-smooth-type}Assume $\lambda$ is
a smooth Keisler measure in $x$ over $M$. Let $p\in S_{\mathcal{U}_{\kappa^{+}}}^{x}\left(M\right)$
such that $\lambda=\nu_{\kappa^{+}}\circ p$ for $\kappa\geq\left|M\right|+\left|T\right|+\left|x\right|$.
Then $p$ is smooth.
\end{prop}

\begin{proof}
Let $p_{1},p_{2}\in S_{\mathcal{U}_{\kappa^{+}}}^{x}\left(N\right)$
be distinct types extending $p$ (for $N\supseteq M$). Again, without
loss of generality $\left|N\right|\leq\kappa$. By Claim \ref{claim:non-conjugate-extensions}
we can choose $p_{1}$ and $p_{2}$ such that for no automorphism
$\sigma$ of $\mathcal{U}_{\kappa^{+}}$, $\sigma\circ p_{1}=p_{2}$.

Let $\lambda'$ the unique extension of $\lambda$ to $N$; in particular
$\lambda'=\nu_{\kappa^{+}}\circ p_{i}$ for $i=1,2$. Since $\ker\left(p_{i}\right)=\left\{ \varphi\left(x,b\right)\mid\lambda'\left(\varphi\left(x,b\right)\right)=0\right\} $,
$\mathcal{B}=L_{x}\left(M\right)/\ker\left(p_{i}\right)$ is independent
of $i$. Let $\mathcal{B}'$ the probability algebra of Borel sets
in $S^{x}\left(M\right)$ up to measure $0$.

We conclude that both $p_{i}$'s can be written as $f_{i}\circ\pi$
where $\pi$ is the projection from $L_{x}\left(N\right)$ to the
algebra $\mathcal{B}$, and $f_{i}:\mathcal{B}\rightarrow Im\left(f_{i}\right)$
are measure preserving embeddings (note $Im\left(f_{i}\right)$ is
a complete subalgebra, see \cite[Proposition 324L]{MR2459668}), and
we can extend them uniquely to $\mathcal{B}'$ by \cite[Proposition 324O and Proposition 323J]{MR2459668},
like in the proof of Corollary \ref{cor:smooth-type-implies-smooth-measure}. 

We get $f_{2}\circ f_{1}^{-1}$ is a partial measure preserving isomorphism
from $Im\left(f_{1}\right)$ into $\mathcal{U}_{\kappa^{+}}$ (which
can, by Proposition \ref{prop:extending-partial-isomorphisms}, be
extended to an automorphism) and $f_{2}\circ f_{1}^{-1}\circ p_{1}=f_{2}\circ\pi=p_{2}$,
contradiction. 
\end{proof}
By \cite[Theorem 322Ca-c]{MR2459668}, every probability algebra is
complete as a Boolean algebra (Dedekind complete, in his terminology). 

Thus by choosing a sufficiently large $\kappa$ (at least $\left(2^{\aleph_{0}}\right)^{+}$)
in Proposition \ref{prop:attaching a general type to a measure} and
using Corollary \ref{cor:smooth-extension} we get, recovering \cite[Proposition 7.9]{pierrebook}:
\begin{cor}
Every Keisler measure over an NIP theory can be extended to a smooth
measure.
\end{cor}

\begin{proof}
Take a Keisler measure $\lambda$ over a model $M$ in $x$. Let $\kappa=\left(2^{\aleph_{0}}+\left|T\right|+\left|x\right|+\left|M\right|\right)^{+}$.
Then by Proposition \ref{prop:attaching a general type to a measure}
there exists a type $p\in S_{\mathcal{U}_{\kappa}}^{x}\left(M\right)$
such that $\lambda=\nu_{\kappa}\circ p$.

By Corollary \ref{cor:smooth-extension} there exists a model $N\supseteq M$
of cardinality at most $2^{\aleph_{0}}+\left|T\right|+\left|x\right|+\left|M\right|$
and a smooth extension $q\in S_{\mathcal{U}_{\kappa}}^{x}\left(N\right)$
of $p$.

Thus letting $\lambda'=\nu_{\kappa}\circ q$, we find $\lambda'$
extends $\lambda$ thus by Corollary \ref{cor:smooth-type-implies-smooth-measure}
we are done.
\end{proof}
\begin{rem}
By \cite[Lemma 7.8]{pierrebook} and the remark following it, if $N\supseteq M$
is an extension and $\lambda$ over $N$ is a smooth extension of
$\mu$ over $M$, then there is some $M\subseteq N'\subseteq N$ such
that $\lambda|_{N'}$ is still smooth and $\left|N'\right|\leq\left|M\right|+\left|T\right|$.

Indeed the lemma and the remark state that $\lambda$ is smooth iff
for every $\varepsilon>0$ and $\phi\left(x,y\right)$ exist $\psi_{i}\left(y\right),\theta_{i}^{1}\left(x\right),\theta_{i}^{2}\left(x\right)$
($i=1,...,n$) over $N$ such that:

1. $\left(\stackrel[i=1]{n}{\bigvee}\psi_{i}\left(y\right)\right)=\top$
.

2. if $b$ is a $y$-tuple in $\mathfrak{C}$ and $\psi_{i}\left(b\right)$
holds then $\theta_{i}^{1}\left(x\right)\rightarrow\phi\left(x,b\right)\rightarrow\theta_{i}^{2}\left(x\right)$.

3. $\lambda\left(\theta_{i}^{2}\left(x\right)\right)-\lambda\left(\theta_{i}^{1}\left(x\right)\right)<\varepsilon$.

Thus we can take an elementary substructure $N'$ of $N$ containing
$M$ and all parameters in the formulas $\theta_{i}^{j},\psi_{i}$
mentioned in the lemma for each $\varepsilon=\frac{1}{n}$ and $\phi\left(x;y\right)$;
then all the formulas are over $N'$ and retain the required properties,
ensuring that $\lambda|_{N'}$ is still smooth.
\end{rem}

The proof of \cite[Lemma 7.8]{pierrebook} relies on the following
fact:
\begin{fact}
\cite[Lemma 7.4]{pierrebook} Let $\phi\left(x,b\right)$ a formula
($b\in\mathfrak{C}$), let $\lambda$ a Keisler measure in $x$ over
$M$. Let:

\begin{align*}
r_{1}=\sup & \left\{ \lambda\left(\varphi\left(x,a\right)\right)\mid\varphi\left(x,a\right)\in L_{x}\left(M\right),\mathfrak{C}\vDash\varphi\left(x,a\right)\rightarrow\phi\left(x,b\right)\right\} \text{ and}\\
r_{2}=\inf & \left\{ \lambda\left(\varphi\left(x,a\right)\right)\mid\varphi\left(x,a\right)\in L_{x}\left(M\right),\mathfrak{C}\vDash\phi\left(x,b\right)\rightarrow\varphi\left(x,a\right)\right\} .
\end{align*}

Then there exists an extension $\lambda'$ of $\lambda$ to $\mathfrak{C}$
such that $\lambda\left(\phi\left(x,b\right)\right)=r$ iff $r_{1}\leq r\leq r_{2}$.
\end{fact}

\begin{rem}
\label{rem:proof-lemma-7.4}For any measure algebra $\mathcal{B}$
and $F\subseteq\mathcal{B}$ there exists a countable $A\subseteq F$
such $A$ and $F$ have the same set of lower bounds, hence $\prod F$
and $\prod A$ exist and are equal (since $\mathcal{B}$ has c.c.c.
and is complete; see for example \cite[Proposition 316E]{MR2459668});
and if $F$ is closed under finite products we can choose $A$ to
be a chain. Thus we conclude $\nu_{\kappa}\left(\prod F\right)=\inf\left\{ \nu_{\kappa}\left(a\right)\right\} _{a\in F}$.
A similar argument shows $\nu_{\kappa}\left(\sum F\right)=\sup\left\{ \nu_{\kappa}\left(a\right)\right\} _{a\in F}$
if $F$ is closed under finite sums.

Let $p\in S_{\mathcal{U}_{\kappa}}^{x}\left(M\right)$, $\phi\left(x,y\right)$
a partitioned formula and $b\in\mathcal{\mathfrak{C}}$ a $y$-tuple.
Then 
\[
\left\{ p\left(\varphi\left(x,a\right)\right)\mid\varphi\left(x,a\right)\in L_{x}\left(M\right),\phi\left(x,b\right)\rightarrow\varphi\left(x,a\right)\right\} 
\]

and 
\[
\left\{ p\left(\varphi_{2}\left(x,a_{2}\right)\right)-p\left(\varphi_{1}\left(x,a_{1}\right)\right)\mid\varphi_{i}\left(x,a_{i}\right)\in L_{x}\left(M\right),\varphi_{1}\left(x,a_{1}\right)\rightarrow\phi\left(x,b\right)\rightarrow\varphi_{2}\left(x,a_{2}\right)\right\} 
\]

are closed under finite products, while
\[
\left\{ p\left(\varphi\left(x,a\right)\right)\mid\varphi\left(x,a\right)\in L_{x}\left(M\right),\varphi\left(x,a\right)\rightarrow\phi\left(x,b\right)\right\} 
\]

is closed under finite sums.
\end{rem}

\begin{rem}
We can also get \cite[Lemma 7.4]{pierrebook} by combining Proposition
\ref{prop:attaching a general type to a measure}, Remark \ref{rem:proof-lemma-7.4}
and Fact \ref{fact:homomorphism-extension}.

Indeed we need only find for any $r_{1}\leq r\leq r_{2}$ some 
\[
\mathfrak{b}\leq\left(\stackrel[\psi\left(x,b\right):\varphi\left(x,a\right)\rightarrow\psi\left(x,b\right)]{}{\prod}p\left(\psi\left(x,b\right)\right)\right)-\left(\stackrel[\psi\left(x,b\right):\psi\left(x,b\right)\rightarrow\varphi\left(x,a\right)]{}{\sum}p\left(\psi\left(x,b\right)\right)\right)
\]
 (for $p$ some type corresponding to $\lambda$) such that $\nu_{\kappa}\left(\mathfrak{b}\right)=r-r_{1}$;
and it is easy to see that when $\kappa\geq\aleph_{0}$, for every
$\mathfrak{a}\in\mathcal{\mathcal{U}_{\kappa}^{+}}$ and every $0\leq s\leq\nu_{\kappa}\left(\mathfrak{a}\right)$
exists $\mathfrak{b}\leq\mathfrak{a}$ such $\nu_{\kappa}\left(\mathfrak{b}\right)=s$
(indeed the same holds for any non-atomic measure alegbra, but here
we can see it directly).
\end{rem}

\section{\label{sec:Analysis of stable case}Analysis of the stable case}

In this section, we analyze the stable and totally transcendental
(t.t.) cases. We show that in the stable case, local Boolean types
are essentially averages of classical $\varphi$-types, and that in
the t.t. case the same is true for complete Boolean types. We start
with the t.t. case, since the general stable case is similar (and
easier, since local ranks are bounded by $\omega$). 

\subsection{The t.t. case}
\begin{defn}
\label{def:support-1}For $q\in S_{\cal B}^{x}\left(A\right)$, let
$\supp\left(q\right)=\set{p\left(x\right)\in S^{x}\left(A\right)}{p\vdash\theta\Rightarrow q\left(\theta\right)>0}$.
\end{defn}

Note that when $q\in S_{\cal B}^{x}\left(A\right)$, $\supp\left(q\right)$
is a closed subset of $S^{x}\left(A\right)$. Also note that if $\Gamma$
is a collection of formulas, closed under finite conjunctions, such
that $q\left(\theta\right)>0$ for all $\theta\in\Gamma$ then there
is a type $p\in\supp\left(q\right)$ such that $p\vdash\Gamma$ (because
$\Gamma\cup\set{\neg\psi}{q\left(\psi\right)=0}$ is consistent). 

 We will use some basic facts from stability theory, namely:
\begin{fact}
\cite[Section 3]{MR1429864} \label{fact:CB analysis}For a topological
space $X$, let $X'=X\backslash\set{x\in X}{x\text{ is isolated}}$.
Assume that $T$ is t.t. and let $X=S^{x}\left(A\right)$ for some
$A\subseteq M\models T$. For $\alpha\in\ord$, let $X^{\left(\alpha\right)}$
be the Cantor-Bendixon analysis of $X$ ($X^{\left(\alpha+1\right)}=\left(X^{\left(\alpha\right)}\right)'$
and $X^{\left(\alpha\right)}=\bigcap_{\beta<\alpha}X^{\left(\beta\right)}$
when $\alpha$ is a limit ordinal). Then for some (successor) $\alpha<\left|T\right|^{+}$,
$X^{\left(\alpha\right)}=\emptyset$. (Note that this definition is
a bit different than the one in \cite[Section 3]{MR1429864}, where
one considers global types.)  
\end{fact}

\begin{rem}
\label{rem:derivative subset}Note that for any topological space
$X$, if $Y\subseteq X$ then $Y'\subseteq X'$.
\end{rem}

\begin{lem}
\label{lem:greater than support}Suppose that $T$ is t.t.. Suppose
that $\cal B$ is any complete Boolean algebra. Let $A\subseteq M\models T$
and let $q\in S_{\cal B}^{x}\left(A\right)$. Let $X=\supp\left(q\right)$.
Let $U$ be the set of all isolated types $r\in X$. For each $r\in U$,
let $\theta_{r}\left(x\right)$ be an isolating formula for $r$.
Then:
\begin{enumerate}
\item $\set{q\left(\theta_{r}\right)}{r\in U}$ is an antichain.
\item For any $\psi\left(x\right)\in L_{x}\left(A\right)$, $q\left(\psi\right)\geq\sum_{r\in U}q\left(\theta_{r}\right)\cdot r\left(\psi\right)$
(where we treat $r$ as a $2$-type).
\end{enumerate}
\end{lem}

\begin{proof}
(1) Note that $0<q\left(\theta_{r}\right)$ for all $r\in U$ since
$r\in\supp\left(q\right)$. Suppose that $r_{1}\neq r_{2}\in U$ and
$0<\mathfrak{b}\leq q\left(\theta_{r_{1}}\right),q\left(\theta_{r_{2}}\right)$.
Hence $0<\mathfrak{b}\leq q\left(\theta_{r_{1}}\land\theta_{r_{2}}\right)$
and thus for some $r\in\supp\left(q\right)$, $r\vdash\theta_{r_{1}}\land\theta_{r_{2}}$
hence $r_{1}=r=r_{2}$ by assumption.

(2) It is enough to show that if $\psi\left(x\right)\in r$ for some
$r\in U$ then $q\left(\psi\right)\geq q\left(\theta_{r}\right)$.
If $\psi\left(x\right)\in r$, it follows that $q\left(\theta_{r}\land\neg\psi\right)=0$
so $q\left(\theta_{r}\right)=q\left(\psi\wedge\theta_{r}\right)\leq q\left(\psi\right)$.
\end{proof}
\begin{thm}
\label{thm:w-stable avg}Suppose that $T$ is totally transcendental,
$A$ is some set, $\cal B$ is a complete Boolean algebra and that
$q\in S_{\cal B}^{x}\left(A\right)$ is a $\cal B$-type. Then there
is a maximal antichain $\sequence{\mathfrak{b}_{r}}{r\in U}$ where
$U\subseteq\supp\left(q\right)$ such that for all $\psi\left(x\right)\in L_{x}\left(A\right)$,
$q\left(\psi\right)=\sum_{r\in U}\mathfrak{b}_{r}\cdot r\left(\psi\right)$.
\end{thm}

\begin{proof}
 For any $0<\mathfrak{b}\in\cal B$, let $\mathcal{B}|_{\mathfrak{b}}$
be the relative algebra. Letting $X=S^{x}\left(A\right)$, we try
to construct a sequence $\sequence{\mathfrak{b}_{\alpha},q_{\alpha},U_{\alpha},\bar{\mathfrak{c}}_{\alpha}}{\alpha<\alpha^{*}}$
for some $\alpha^{*}\leq\left|T\right|^{+}$ such that:
\begin{itemize}
\item $0<\mathfrak{b}_{\alpha}\leq\mathfrak{b}_{\beta}$ for $\beta<\alpha$;
$\mathfrak{b}_{0}=1$ and more generally $\mathfrak{b}_{\alpha}=-\sum_{\beta<\alpha,r\in U_{\beta}}\mathfrak{c}_{\beta,r}$;
$q_{\alpha}\in S_{\mathcal{B_{\alpha}}}^{x}\left(A\right)$ where
$\mathcal{B}_{\alpha}=\mathcal{B}|_{\mathfrak{b}_{\alpha}}$; $q_{\alpha}\left(\psi\right)=q\left(\psi\right)\cdot\mathfrak{b}_{\alpha}$
for any $\psi\in L_{x}\left(A\right)$; $U_{\alpha}\subseteq\supp\left(q_{\alpha}\right)\subseteq X^{\left(\alpha\right)}$;
$\bar{\mathfrak{c}}_{\alpha}=\sequence{\mathfrak{c}_{\alpha,r}}{r\in U_{\alpha}}$
is an antichain contained in $\mathcal{B}_{\alpha}$;  $q_{\alpha}\left(\psi\right)\geq\sum_{r\in U_{\alpha}}\mathfrak{c}_{\alpha,r}\cdot r\left(\psi\right)$
for any $\psi\in L_{x}\left(A\right)$. 
\end{itemize}
Given $\sequence{\mathfrak{b}_{\beta},q_{\beta},U_{\beta},\bar{\mathfrak{c}}_{\beta}}{\beta<\alpha}$,
if $\sum_{\beta<\alpha,r\in U_{\beta}}\mathfrak{c}_{\beta,r}=1$ we
stop and let $\alpha^{*}=\alpha$. Otherwise, let $\mathfrak{b}_{\alpha}$,
$q_{\alpha}$ as above and let $U_{\alpha}\subseteq\supp\left(q_{\alpha}\right)$
the set of isolated types in $\supp\left(q_{\alpha}\right)$ (so for
$\alpha=0$, $q_{0}=q$ and $\mathfrak{b}_{0}=1$). For $r\in U_{\alpha}$,
let $\mathfrak{c}_{\alpha,r}=q_{\alpha}\left(\theta_{r}\right)$ where
$\theta_{r}$ isolates $r$. By Lemma \ref{lem:greater than support}
(1), $\set{\mathfrak{c}_{\alpha,r}}{r\in U_{\alpha}}$ is an antichain
in $\cal B_{\alpha}$. Note that $q_{\alpha}\left(\psi\right)\geq\sum_{r\in U_{\alpha}}\mathfrak{c}_{\alpha,r}\cdot r\left(\psi\right)$
by Lemma \ref{lem:greater than support} (2). Now prove by induction
on $\alpha$ that $\supp\left(q_{\alpha}\right)\subseteq X^{\left(\alpha\right)}$
(this follows from Remark \ref{rem:derivative subset} and the fact
that $\supp\left(q_{\alpha+1}\right)\subseteq\supp\left(q_{\alpha}\right)'$)
and that $\set{\mathfrak{c}_{\beta,r}}{r\in U_{\beta},\beta<\alpha}$
is an antichain in $\cal B$.

Finally, since for some $\beta<\left|T\right|^{+}$, $X^{\left(\beta\right)}=\emptyset$,
it follows that $\alpha^{*}\leq\beta$ (otherwise, $\supp\left(q_{\beta}\right)=\emptyset$
and so $\mathfrak{b}_{\beta}=q_{\beta}\left(x=x\right)=0$, contradiction).
Hence for all $\psi\in L_{x}\left(A\right)$, $q\left(\psi\right)\geq\sum_{\alpha<\alpha^{*}}q_{\alpha}\left(\psi\right)\geq\sum_{r\in U_{\alpha},\alpha<\alpha^{*}}\mathfrak{c}_{\alpha,r}\cdot r\left(\psi\right)$
and $\set{\mathfrak{c}_{\alpha,r}}{\alpha<\alpha^{*},r\in U_{\alpha}}$
is a maximal antichain in $\cal B$. Since this is also true for $\neg\psi$,
we have equality and we are done. 
\end{proof}

\subsection{The stable case}

Fix a partitioned formula $\varphi\left(x,y\right)$ in some theory
$T$, and let $A\subseteq\C$. As in \cite[Section 2]{MR1429864},
by a\emph{ $\varphi$-formula over $A$}, we will mean a formula $\psi\left(x\right)\in L_{x}\left(A\right)$
which is equivalent to a Boolean combination of instances of $\varphi$
over $A$ (over a model $M$, a $\varphi$-formula is just a Boolean
combination of instances of $\varphi$ over $M$). Let $L_{\varphi,x}\left(A\right)$
be the Boolean algebra of $\varphi$-formulas over $A$ up to equivalence
in $\C$. Let  $S_{\varphi}^{x}\left(A\right)$ be the set of all
complete $\varphi$-types over $A$ in $x$, i.e., maximal consistent
sets of $\varphi$-formulas over $A$.  
\begin{defn}
\label{def:support}(local Boolean type) Suppose that $\cal B$ is
a Boolean algebra and $\varphi\left(x,y\right)$ is a partitioned
formula. A \emph{$\cal B,\varphi$-type} over a set $A$ is a homomorphism
from $L_{\varphi,x}\left(A\right)$ to $\mathcal{B}$. Denote the
set of $\cal B,\varphi$-types by $S_{\cal B,\varphi}^{x}\left(A\right)$.
For $q\in S_{\cal B,\varphi}^{x}$ let $\supp_{\varphi}\left(q\right)=\set{p\left(x\right)\in S_{\varphi}^{x}\left(A\right)}{p\vdash\theta\Rightarrow q\left(\theta\right)>0}$.

Similarly to the previous section, we have:
\end{defn}

\begin{fact}
\label{fact:stable CB rank}\cite[Section 3]{MR1429864} Assume that
$\varphi\left(x,y\right)$ is stable in some complete theory $T$,
$A\subseteq M\models T$ and let $X=S_{\varphi}^{x}\left(A\right)$.
Then for some $n<\omega$, $X^{\left(n+1\right)}=\emptyset$. 
\end{fact}

\begin{thm}
\label{thm:avg of types}Suppose that $\varphi\left(x,y\right)$ is
stable, $A$ is some set, $\cal B$ is a complete Boolean algebra
and that $q\in S_{\cal B,\varphi}^{x}\left(A\right)$ is a $\cal B,\varphi$-type.
Then there is a maximal antichain $\sequence{\mathfrak{b}_{r}}{r\in U}$
where $U\subseteq\supp_{\varphi}\left(q\right)$ such that for all
$\psi\left(x\right)\in L_{\varphi,x}\left(A\right)$, $q\left(\psi\right)=\sum_{r\in U}\mathfrak{b}_{r}\cdot r\left(\psi\right)$. 
\end{thm}

\begin{proof}
The proof is exactly as the proof of \ref{thm:w-stable avg}, working
with $X=S_{\varphi}^{x}\left(A\right)$ and with local Boolean types,
Replacing Fact \ref{fact:CB analysis} with Fact \ref{fact:stable CB rank}.
We leave the details to the reader. 
\end{proof}
\begin{rem}
When $T$ is stable and $q\in S_{\cal B}^{x}\left(A\right)$, this
essentially means that $q|_{\varphi}$ factors through $2^{\left|U\right|}\hookrightarrow\stackrel[r\in U]{}{\prod}\mathcal{B}|_{\mathfrak{b}_{r}}\cong\mathcal{B}$,
see \cite[Proposition 6.4]{MR991565}. In particular we get again,
more directly, that for $\mathcal{B}$ which is $\kappa$-c.c., $\left|Im\left(q|_{\varphi}\right)\right|\leq2^{<\kappa}$
(see Proposition \ref{prop:image not too big}). When $T$ is t.t.,
we get similarly that $\left|Im\left(q\right)\right|\leq2^{<\kappa}$. 
\end{rem}

\subsection{Non-forking}

Using (the proof of) Theorem \ref{thm:avg of types}, one can recover
the theory of forking in stable theories. 
\begin{defn}
Let $\cal B$ be any Boolean algebra and let $T$ be any theory. Let
$A\subseteq B$ be any sets . Say that a $\cal B$-type or a $\cal B,\varphi$-type
$q$ \emph{forks over $A$} if for some formula $\theta\left(x\right)$
over $B$ which forks over $A$, $q\left(\theta\right)>0$. (For the
definition of forking see e.g., \cite[Definition 7.1.7]{TentZiegler}). 
\end{defn}

\begin{fact}
\label{fact:stationarity}(E.g., \cite[Section 2]{MR1429864}) If
$M\prec N$, $\varphi\left(x,y\right)$ is stable, then any $\varphi$-type
$p\in S_{\varphi}^{x}\left(M\right)$ has a unique non-forking extension
$p|_{N}$ to $S_{\varphi}\left(N\right)$. The same is true assuming
elimination of imaginaries when $M$ is replaced by an algebraically
closed set $A$. 
\end{fact}

\begin{rem}
\label{rem:existence of nf} Suppose that $p\in S_{\cal B}^{x}\left(B\right)$
does not fork over $A\subseteq B$. Then there is a global non-forking
(over $A$) extension $q\in S_{\cal B}^{x}\left(\C\right)$ (and the
same is true for local Boolean types). This follows by the fact that
forking formulas over $A$ form an ideal and using Fact \ref{fact:function-extension}.
 When $B$ is a model $M$ then $p$ does not fork over $M$ and
if $T$ is stable (or even simple) then this is true in general. 

Suppose that $\varphi\left(x,y\right)$ is stable and $p\in S_{\cal B,\varphi}\left(M\right)$
for some model $M\models T$. Then we can find an explicit extension:
by Theorem \ref{thm:avg of types}, we can write $p$ as the sum $\sum_{r\in U}\mathfrak{b}_{r}\cdot r$
for some maximal antichain $U\subseteq\supp_{\varphi}\left(p\right)$
and we let $q=\sum_{r\in U}\mathfrak{b}_{r}\cdot r|_{\C}$ (where
$r|_{\C}$ is the unique global non-forking extension of $r$). A
similar statement holds in the t.t. case. 
\end{rem}

Next we would like to prove that there is a unique non-forking extension. 
\begin{lem}
\label{lem:greater than support nf}Suppose that $\varphi\left(x,y\right)$
is stable and $\cal B$ is any complete Boolean algebra. Let $M\prec N\models T$
and let $q\in S_{\cal B,\varphi}\left(N\right)$ be non-forking over
$M$. Let $X=\supp_{\varphi}\left(q|_{M}\right)$ ($q|_{M}$ is the
restriction of $q$ to $L_{\varphi,x}\left(M\right)$). Let $U$ be
the set of all isolated $\varphi$-types $r\in X$. For each $r\in U$,
let $\theta_{r}\left(x\right)$ be an isolating formula for $r$ (so
it is a $\varphi$-formula). Then 
\begin{itemize}
\item For any $\psi\left(x\right)\in L_{\varphi,x}\left(N\right)$, $q\left(\psi\right)\geq\sum_{r\in U}q\left(\theta_{r}\right)\cdot r|_{N}\left(\psi\right)$
(where $r|_{N}$ is the unique non-forking extension of $r$ to $N$).
\end{itemize}
\end{lem}

\begin{proof}
It is enough to show that $q\left(\psi\right)\geq q\left(\theta_{r}\right)$
when $\psi\in r|_{N}$. Suppose that $\psi\left(x\right)\in r|_{N}$
but $q\left(\theta_{r}\land\neg\psi\right)>0$. Then there is some
$r'\in\supp_{\varphi}\left(q\right)$ such that $r'$ contains $\theta_{r}\land\neg\psi$.
But then $r'$ does not fork over $M$ (because $q$ does not fork
over $M$) and $r'|_{M}$ contains $\theta_{r}$ and is in $X$ and
thus $r'|_{M}=r|_{M}$ and by Fact \ref{fact:stationarity}, $r'=r$.
\end{proof}
\begin{thm}
\label{thm:stationary bool types}Suppose that $\varphi\left(x,y\right)$
is stable. Suppose that $n<\omega$. Then, whenever $\cal B$ is a
complete Boolean algebra, $M\prec N$ and $q\in S_{\cal B,\varphi}\left(N\right)$
does not fork over $M$,  there is a maximal antichain $\sequence{\mathfrak{b}_{r}}{r\in U}$
where $U\subseteq\supp_{\varphi}\left(q|_{M}\right)$ which depends
only on $q|_{M}$ such that for all $\psi\left(x\right)\in L_{\varphi,x}\left(N\right)$,
$q\left(\psi\right)=\sum_{r\in U}\mathfrak{b}_{r}\cdot r|_{N}\left(\psi\right)$.
In particular, $q$ is the unique non-forking extension of $q|_{M}$.
\end{thm}

\begin{proof}
The proof follows the same lines as in the proof of Theorem \ref{thm:avg of types}
(and Theorem \ref{thm:w-stable avg}), using Lemma \ref{lem:greater than support nf}
instead of Lemma \ref{lem:greater than support}. 
\end{proof}
\begin{rem}
\label{rem:algebraically closed stationary}As in the classical case,
we can extend these results (existence and uniqueness of non-forking
extensions) for an arbitrary algebraically closed set $A$, assuming
elimination of imaginaries. 
\end{rem}

\subsection{Connection to Keisler measures}

Using the general results on Boolean types we can recover and prove
some results on Keisler measures. The following result appeared in
\cite[Fact 1.1]{PillayDomination}, \cite[Fact 2.2]{ArtemKyle} for
models. 
\begin{cor}
\label{cor:measure-sum-of-types}Suppose that $\varphi\left(x,y\right)$
is stable and that $\mu$ is a Keisler measure on $L_{\varphi,x}\left(A\right)$
for some set $A$. Then there is a countable family $\sequence{p_{i}}{i<\omega}$
of complete $\varphi$-types over $A$ and positive real numbers $\sequence{\alpha_{i}}{i<\omega}$
such that $\sum_{i<\omega}\alpha_{i}=1$ and for any $\psi\left(x\right)\in L_{\varphi,x}\left(A\right)$,
$\mu\left(\psi\right)=\sum\alpha_{i}p_{i}\left(\psi\right)$. 

Similarly, if $T$ is t.t. and $\mu$ is a Keisler measure on $L_{x}\left(A\right)$
then there is a countable family $\sequence{p_{i}}{i<\omega}$ of
complete types over $A$ and positive real numbers $\sequence{\alpha_{i}}{i<\omega}$
such that $\sum_{i<\omega}\alpha_{i}=1$ and for any $\psi\left(x\right)\in L_{x}\left(A\right)$,
$\mu\left(\psi\right)=\sum\alpha_{i}p_{i}\left(\psi\right)$. 
\end{cor}

\begin{proof}
Given $\mu$, let $\cal B$ be the Boolean algebra of Borel subsets
of $S_{\varphi}^{x}\left(A\right)$ up to $\mu$-measure $0$ (recall
Remark \ref{rem:generation-of-Keisler}) and let $q\in S_{\cal B,\varphi}^{x}\left(A\right)$
be the natural homomorphism from $\varphi$-formulas over $A$ (up
to equivalence over $\C$) to $\cal B$. Now apply Theorem \ref{thm:avg of types}
to $q$ and $\cal B$ to obtain a maximal antichain $\set{\mathfrak{b}_{r}}{r\in U}$
where $U\subseteq\supp_{\varphi}\left(q\right)$ such that for all
$\psi\left(x\right)\in L_{x}\left(A\right)$, $q\left(\psi\right)=\sum_{r\in U}\mathfrak{b}_{r}\cdot r\left(\psi\right)$.
Note that $U$ must be countable as $\cal B$ is c.c.c. Letting $\alpha_{r}=\mu\left(\mathfrak{b}_{r}\right)$
we are done. The second statement follows similarly from Theorem \ref{thm:w-stable avg}.
\end{proof}
\begin{defn}
A Keisler measure $\mu$ on $L_{x}\left(N\right)$ does not fork over
$M\prec N$ if whenever $\mu\left(\theta\right)>0$, $\theta$ does
not for over $M$. 
\end{defn}

\begin{cor}
Suppose that $\varphi\left(x,y\right)$ is stable and that $\mu$
is a Keisler measure on $L_{\varphi}\left(M\right)$ for some model
$M$. Then $\mu$ has a unique global non-forking extension to $\C$.
More generally, this holds when replacing $M$ by any algebraically
closed set $A$, assuming elimination of imaginaries. 
\end{cor}

\begin{proof}
We use a local version of Proposition \ref{prop:attaching a general type to a measure}:
given $\mu$, we can find $p\in S_{\varphi,\cal U_{\kappa}}^{x}\left(M\right)$
such that $\mu=\nu_{\kappa}\circ p$ for $\kappa=\left|T\right|+\left|M\right|$.
By Remark \ref{rem:existence of nf} there is a non-forking extension
$q\in S_{\varphi,\cal U_{\kappa}}^{x}\left(N\right)$ and then we
can define $\mu'=\mu\circ q$. For uniqueness, suppose that $\lambda_{1},\lambda_{2}$
are two non-forking measures over $N\succ M$ extending $\mu$. We
may assume $\left|N\right|=\left|M\right|$ and let $\kappa$ be as
above. Let $q_{1},q_{2}\in S_{\varphi,\cal U_{\kappa^{+}}}^{x}\left(N\right)$
be corresponding $\varphi,\cal U_{\kappa^{+}}$-types. By a local
version of Lemma \ref{lem:extend-embedding} we may assume that both
$q_{1},q_{2}$ extend $p$. Thus we are done by Theorem \ref{thm:stationary bool types}.
The more general statement follows similarly by Remark \ref{rem:algebraically closed stationary}.
\end{proof}
\begin{rem}
Note that in the context of the first part of Corollary \ref{cor:measure-sum-of-types}
where $A$ is a model and $N\supseteq A$, the unique non-forking
extension of $\mu$ to $N$ is the weighted sum $\sum\alpha_{i}p_{i}|_{N}$
where $p_{i}|_{N}$ is the unique non-forking extension of $p_{i}$
to $N$. This follows immediate from the fact that the sum does not
fork. The analogous result holds in the t.t. case. \bibliographystyle{alpha}
\bibliography{0C__Users_Ori_Segel_Dropbox_Math_Project_common2}
\end{rem}

\end{document}